\RequirePackage[clockwise]{rotating}
\documentclass[a4paper, reqno]{amsart}
\DeclareRobustCommand{\SkipTocEntry}[5]{}
\usepackage{amsmath,amsthm,amssymb,relsize}
\usepackage[nobysame,alphabetic,initials]{amsrefs}
\usepackage[margin=30mm]{geometry}
\usepackage{mathrsfs,mathtools,makecell}

\usepackage{booktabs, multirow, adjustbox, array}
\newcommand{\bighash}[1]{\,\,{\scalebox{1.2}{\#}^{#1}}}
\newcolumntype{R}[2]{%
    >{\adjustbox{angle=#1,lap=\width-(#2)}\bgroup}%
    l%
    <{\egroup}%
}
\newcommand*\tilt{\multicolumn{1}{R{30}{1em}}}
\newcommand{\ra}[1]{\renewcommand{\arraystretch}{#1}}
\usepackage{pifont}
\newcommand{\cmark}{\ding{51}}%
\newcommand{\xmark}{\ding{55}}%
\usepackage{floatrow}
\DeclareFloatFont{tiny}{\small}
\usepackage{enumerate}

\usepackage{tikz}
\usetikzlibrary{
  cd,
  calc,
  positioning,
  fit,
  arrows,
  decorations.pathreplacing,
  decorations.markings,
  shapes.geometric,
  backgrounds,
  bending
}
\usepackage{tikzsymbols}
\usepackage{xcolor}

\definecolor{darkblue}{rgb}{0,0,0.6}
\newcommand{\purple}{\textcolor{purple}}
\newcommand{\red}{\textcolor{red}}
\newcommand{\blue}{\textcolor{blue}}

\usepackage[breaklinks, pdftex, ocgcolorlinks,colorlinks=true, citecolor=darkblue, filecolor=darkblue, linkcolor=darkblue, urlcolor=darkblue]{hyperref}
\usepackage[capitalize,noabbrev]{cleveref}

\makeatletter
\newtheorem*{rep@theorem}{\rep@title}
\newcommand{\newreptheorem}[2]{%
\newenvironment{rep#1}[1]{%
 \def\rep@title{#2 \ref{##1}}%
 \begin{rep@theorem}}%
 {\end{rep@theorem}}}
\makeatother

\newtheorem{proposition}{Proposition}[section]
\newtheorem{theorem}[proposition]{Theorem}

\newtheorem{lemma}[proposition]{Lemma}

\theoremstyle{definition}
\newtheorem{definition}[proposition]{Definition}
\newtheorem{question}[proposition]{Question}
\newtheorem{example}[proposition]{Example}

\theoremstyle{remark}
\newtheorem{remark}[proposition]{Remark}

\newtheorem*{remark*}{Remark}

\newreptheorem{theorem}{Theorem}
\newreptheorem{lemma}{Lemma}
\newreptheorem{proposition}{Proposition}
\newreptheorem{corollary}{Corollary}
\newreptheorem{question}{Question}
\numberwithin{equation}{section}

\newcommand{\Diff}{\text{Diff}}

\newcommand{\sign}{\operatorname{sign}}

\newcommand{\Q}{\mathbb{Q}}

\newcommand{\RR}{\mathbb{R}}
\newcommand{\R}{\mathbb{R}}
\newcommand{\Z}{\mathbb{Z}}

\newcommand{\cN}{\mathcal{N}}

\newcommand{\im}{\operatorname{Im}}
\newcommand{\Id}{\operatorname{Id}}

\newcommand{\Arf}{\operatorname{Arf}}

\newcommand{\pt}{*}

\newcommand{\ol}{\overline}
\newcommand{\wt}{\widetilde}
\newcommand{\wh}{\widehat}
\newcommand{\sm}{\setminus}

\newcommand{\ks}{\operatorname{ks}}

\newcommand{\CP}{\mathbb{CP}}
\newcommand{\RP}{\mathbb{RP}}
\renewcommand{\star}{\mathop{}\!*}
\DeclareMathOperator{\capp}{cap}

\DeclareMathOperator{\Aut}{Aut}
\DeclareMathOperator{\hAut}{hAut}
\DeclareMathOperator{\Spin}{Spin}
\DeclareMathOperator{\BSpin}{BSpin}
\DeclareMathOperator{\TOPSpin}{TOPSpin}
\DeclareMathOperator{\Wh}{Wh}
\DeclareMathOperator{\Int}{Int}
\DeclareMathOperator{\GL}{GL}
\DeclareMathOperator{\E}{E}
\DeclareMathOperator{\K}{K}
\DeclareMathOperator{\BTOP}{BTOP}
\newcommand{\TOP}{\mathrm{TOP}}
\newcommand{\PL}{\mathrm{PL}}
\newcommand{\OO}{\mathrm{O}}
\newcommand{\G}{\mathrm{G}}

\DeclareMathOperator{\pr}{pr}

\newcommand{\bsm}{\left(\begin{smallmatrix}}
\newcommand{\esm}{\end{smallmatrix}\right)}

\usepackage{letltxmacro}

\LetLtxMacro\Oldfootnote\footnote

\begin{document}
\title{Counterexamples in $4$-manifold topology}

\author{Daniel Kasprowski}
\address{Rheinische Friedrich-Wilhelms-Universit\"at Bonn, Mathematisches Institut,\newline\indent Endenicher Allee 60, 53115 Bonn, Germany}
\email{kasprowski@uni-bonn.de}

\author{Mark Powell}
\address{School of  Mathematics and Statistics, University of Glasgow, United Kingdom}
\email{mark.powell@glasgow.ac.uk}

\author{Arunima Ray}
\address{Max Planck Institut f\"{u}r Mathematik, Vivatsgasse 7, 53111 Bonn, Germany}
\email{aruray@mpim-bonn.mpg.de }

\def\subjclassname{\textup{2020} Mathematics Subject Classification}
\expandafter\let\csname subjclassname@1991\endcsname=\subjclassname

\begin{abstract}
We illustrate the rich landscape of 4-manifold topology through the lens of counterexamples.
We consider several of the most commonly studied equivalence relations on $4$-manifolds and how they are related to one another.  We explain implications e.g.\ that $h$-cobordant manifolds are stably homeomorphic, and we provide examples illustrating the failure of other potential implications.  The information is conveniently organised in a flowchart and a table.
\end{abstract}
\maketitle
\addtocontents{toc}{\SkipTocEntry}
\section{Introduction}

The goal of this paper is to organise various equivalence relations in $4$-manifold topology, and to understand the connections between them.
We consider closed, connected $4$-manifolds, unless otherwise specified, and we work in both the smooth and topological settings. Much work on 4-manifolds focusses on exotic behaviour, e.g.\ $4$-manifolds that are homeomorphic but not diffeomorphic.   We aim to illustrate, more broadly, the wealth of $4$-manifold topology that has been discovered. The flowchart in \cref{fig:flowchart} shows the relationships between the equivalence relations we study.
We will recall their definitions in \cref{section:definitions-of-relations}, and prove the nontrivial implications in \cref{section:just-of-implications}.

\begin{figure}[htb]
\includegraphics[width=15cm]{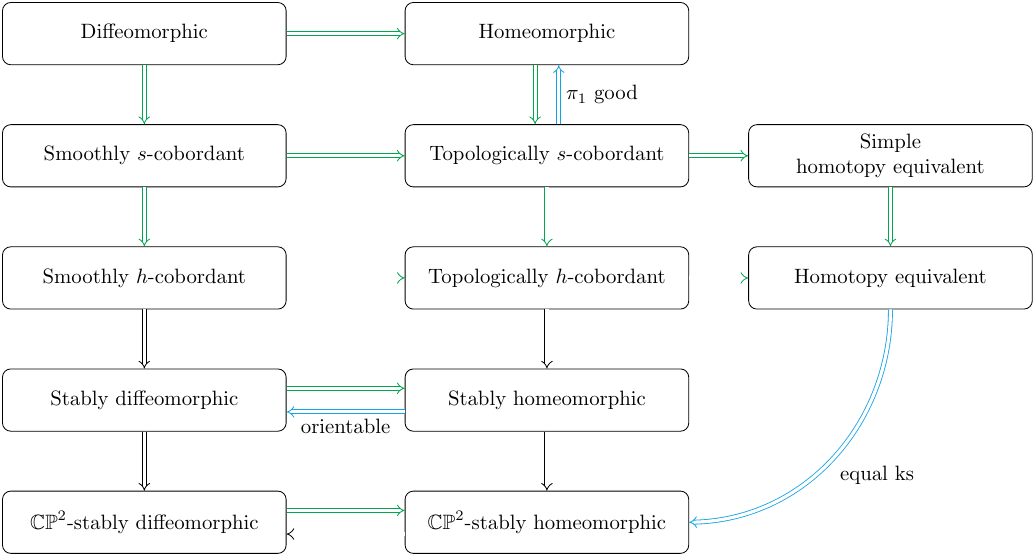}
\caption{Equivalence relations on $4$-manifolds. The implications shown in green are immediate. The black and blue implications are proven in~\cref{section:just-of-implications}. The blue implications hold when the corresponding condition is true, e.g.\ homotopy equivalent $4$-manifolds are $\CP^2$-stably homeomorphic if their Kirby-Siebenmann invariants coincide.
 Where necessary for an implication to make sense, we assume that the manifolds are smooth. For example, the bottom arrow means that closed, \emph{smooth} $\CP^2$-stably homeomorphic $4$-manifolds are $\CP^2$-stably diffeomorphic, since the latter notion is only defined for smooth $4$-manifolds.
}
\label{fig:flowchart}
\end{figure}

We collect counterexamples to the converses of the majority of the implications shown.
Most of the results we discuss are known in the literature, although there are some original observations and results.


The behaviour and study of $4$-manifolds is qualitatively different from that in other dimensions. In lower dimensions, much can be said using tools from geometry, perhaps best exemplified by the geometrisation theorem~\cites{Thurston-Geom-conj,Perelman:2002-1,Perelman:2003-2,Perelman:2003-1}.  Higher-dimensional manifolds are studied via homotopy theoretic and algebraic methods, thanks to the Whitney trick and the powerful tools of surgery theory~\cites{Browder,Novikov,Sullivan,Wall-surgery-book,Kirby-Siebenmann:1977-1} and the $s$-cobordism theorem~\cites{Smale:1962-1,Barden,Mazur:1963-1,Stallings-Tata,Kirby-Siebenmann:1977-1}. In dimension four, the Whitney trick does not directly apply, and surgery and the $s$-cobordism theorem are only available under special circumstances.



The first major progress on the classification of 4-manifolds was due to Whitehead and Milnor~\cites{Whitehead-4-complexes,Milnor-simply-connected-4-manifolds}, who classified simply connected 4-manifolds up to homotopy equivalence.
The homotopy classification has since been completed for more fundamental groups, and remains a topic of active research \cites{Hambleton-Kreck:1988-1,KKR92,teichnerthesis,Baues-Bleile,KNR,KPR-homotopy}.
A diffeomorphism classification was, and still remains, out of reach. Prior to the 1980s, progress on purely topological $4$-manifolds was impossible, in the absence of fundamental tools like topological transversality. Instead, Wall~\cites{Wall-diffeos-4-manifolds,Wall-on-simply-conn-4mflds} and Cappell-Shaneson~\cites{CS-new-four-mflds-bams,CS-new-four-mflds-annals,CS-on-4-d-surgery} studied 4-manifolds up to stable diffeomorphism, namely up to connected sum with copies of $S^2 \times S^2$. In particular, Wall gave the stable classification of simply connected 4-manifolds. As with the homotopy classification, using Kreck's ideas from  \cite{surgeryandduality} the stable classification has since been completed for more fundamental groups, and remains a topic of active research \cites{teichnerthesis,HK93b,CAR-95,Spaggiari-03,HKT,KLPT,KPT,KPT21,Debray}.
The following examples compare the equivalence relations of homotopy equivalence and stable diffeomorphism.

\begin{example}\label{intro-example-CS-K-Ak}
  The pairs of manifolds presented in \cref{subsection:row-5,subsection:row-6,subsection:row-12}, due to Kreck~\cite{Kreck-exotic}, Cappell-Shaneson~\cites{CS-new-four-mflds-bams,CS-new-four-mflds-annals}, and Akbulut~\cites{Akbulut-on-fake,akbulut-fake-gluck,Akbulut-fake-4-manifold}  respectively, are smooth, nonorientable $4$-manifolds that are simple homotopy equivalent (in fact they are now known to be homeomorphic~\cites{F,HKT-nonorientable,Wang-nonorientable}) but not stably diffeomorphic. By a result of Gompf~\cite{Gompf-stable} (see~\cref{thm:gompf-stable}) smooth, orientable $4$-manifolds that are (stably) homeomorphic are stably diffeomorphic, so it is inevitable that these examples are nonorientable.
\end{example}

\begin{example}\label{intro-example:teichner}
    The examples of Teichner~\cite{Teichner:1997-1} from \cref{subsection:row-11} provide smooth, orientable $4$-manifolds that are simple homotopy equivalent but not stably homeomorphic. These can be used to construct arbitrarily large collections which have these properties pairwise. We show in \cref{prop:finite-stab-homeo} that every such collection is finite.
\end{example}

\begin{example}\label{intro-example:leech}
 In \cref{subsection:row-9} we discuss two closed, orientable, simply connected topological $4$-manifolds that are stably homeomorphic but not homotopy equivalent, because they have inequivalent intersection pairings.  Proposition~\ref{prop:simply-conn-hom-equiv-iff-stably-diff} explains that such a phenomenon cannot occur for smooth, simply connected, closed 4-manifolds.
\end{example}

\begin{example}\label{intro-example:kreckschafer}
  The examples of Kreck--Schafer~\cite{Kreck-Schafer} discussed in \cref{subsection:row-10} are smooth, orientable $4$-manifolds (with nontrivial fundamental groups) that are stably diffeomorphic, but not homotopy equivalent. They also have isometric intersection pairings. 
\end{example}

The constructions of the manifolds mentioned in \cref{intro-example:teichner,intro-example:leech} use Freedman's work, which we discuss presently.  First we highlight the following open question comparing stable diffeomorphism and homotopy equivalence.

\begin{question}
  Are there arbitrarily large families of smooth 4-manifolds that are all stably diffeomorphic but pairwise homotopically inequivalent?  Or even better, infinite sets with this property?
\end{question}

The early 1980s saw Freedman's work~\cite{F} showing that the Whitney trick can be applied in ambient dimension four under certain conditions, establishing the exactness of the surgery sequence and the $s$-cobordism theorem with some restrictions on the fundamental group~\cite{FQ}. See \cref{section:just-of-implications,Section:review-surgery}
for further details. Combined with subsequent work of Quinn~\cite{Quinn-annulus}, Freedman's theorem made it possible to upgrade the homotopy classification, the stable classification, or both, to homeomorphism classifications; see for example~\cites{F,FQ,HK88b,HK93b,HKT-nonorientable,Wang-nonorientable,HKT}.

 It is straightforward to see that homeomorphism implies homotopy equivalence, for general spaces.
We now explain a sequence of counterexamples to the converse for 4-manifolds, i.e.\ pairs of 4-manifolds that are homotopy equivalent but not homeomorphic. Along the way we illustrate our approach to investigating counterexamples.
Namely, while investigating the failure of the converses of the implications in the flowchart, we will progressively impose restrictions on the counterexamples, e.g.\ that they be smooth, or orientable.

\begin{example}\label{intro-example-starCP2}
The well-known pair $\CP^2$ and $\star\CP^2$~\cite{F} (see~\cref{subsection:row-3}) are homotopy equivalent but not homeomorphic. The latter manifold, sometimes called the \emph{Chern manifold}, was constructed by Freedman and is homotopy equivalent to $\CP^2$, but not homeomorphic to it.
Indeed, $\CP^2$ and $\star\CP^2$ have unequal Kirby-Siebenmann invariants, implying that $\star\CP^2$ is not smoothable.
\end{example}

The natural question then arises whether there exists a pair of smooth, closed $4$-manifolds that are homotopy equivalent but not homeomorphic. Freedman's classification result~\cites{F,FQ} implies that there is no such pair of simply connected $4$-manifolds.

\begin{example}\label{intro-example-star-RP4}
A pair that satisfies our new demand consists of $\RP^4 \# \CP^2$ and $\mathcal{R} \# \star\CP^2$~\cites{ruberman84,HKT-nonorientable,Ruberman-Stern} (see~\cref{subsection:row-4}), where $\mathcal{R}$ is a 4-manifold homotopy equivalent to $\RP^4$ but with nontrivial Kirby-Siebenmann invariant. However, $\RP^4 \# \CP^2$ and $\mathcal{R} \# \star\CP^2$ are nonorientable.
\end{example}

We can then escalate further to ask for a pair of smooth, closed, orientable $4$-manifolds that are homotopy equivalent but not homeomorphic. 

\begin{example}\label{intro-example-LxS1}
The simplest such example we know of comes from Turaev~\cite{Turaev1988} (see~\cref{subsection:row-7}), who showed that for lens spaces $L$ and $L'$ that are homotopy equivalent but not homeomorphic, the same holds for the $4$-manifolds $L \times S^1$ and $L' \times S^1$.
\end{example}

Finally, one may ask for an infinite collection of closed, smooth, orientable $4$-manifolds that are homotopy equivalent but not homeomorphic. To our knowledge, this is an open question. However, the following example answers the question for topological 4-manifolds.

\begin{example}\label{intro-thm-1}
Let $M := L \times S^1$, where $L$ is a lens space $L_{p,q}$ with $p \geq 2$, $1\leq q < p$, and $(p,q)=1$. Then Kwasik-Schultz~\cite{Kwasik-Schultz}*{Theorem~1.2} constructed an infinite collection of closed, orientable, topological $4$-manifolds $\{M_i\}_{i=1}^\infty$, that are all simple homotopy equivalent to $M$ but pairwise not homeomorphic.
The proof of Kwasik and Schultz relies on higher $\rho$-invariants. In \cref{subsection:row-13} we provide a different argument via the surgery exact sequence that enables us to establish other properties of these manifolds. For example one can assume that they are all stably homeomorphic and are pairwise not $h$-cobordant.
\end{example}

We will also show (or give citations showing) that all of the pairs of 4-manifolds discussed in \cref{intro-example-starCP2,intro-example-star-RP4,intro-example-LxS1} are simple homotopy equivalent, and that the pairs from \cref{intro-example-star-RP4,intro-example-LxS1} are stably homeomorphic.

As part of surgery programmes to classify 4-manifolds, the relations of simply homotopy equivalence, $h$-cobordism, and $s$-cobordism are prominent. All are necessary conditions for homeomorphism.
The following theorem compares these three relations. It is the main original result of the article.

\begin{theorem}\label{intro-thm-2}
  For every $n\geq 1$, there is a collection ~$\{N_i\}_{i=1}^n$ of closed, orientable, topological 4-manifolds, that are all simple homotopy equivalent and $h$-cobordant to one another, but which are pairwise not $s$-cobordant.
\end{theorem}

Our proof makes use of a braid of exact sequences due to Hambleton-Kreck~\cite{Hambleton-Kreck-HSEs} which enables one to estimate the size of the group of homotopy automorphisms of the $4$-manifold~$L_{p,q} \times S^1$, where $L_{p,q}$ is a lens space. We combine this with the surgery exact sequence for 4-manifolds with fundamental group $\Z/p \times \Z$ to construct our families of examples.

\begin{question}
Is there a pair of smooth $4$-manifolds that are simple homotopy equivalent and $h$-cobordant, but not $s$-cobordant?
If so, what is the largest possible cardinality of such a collection of $4$-manifolds?
\end{question}

One should try to answer this question with the strongest possible assumptions on categories for the $h$- and $s$-cobordisms.

Finally, at opposite ends of the flowchart in~\cref{fig:flowchart}, one can compare with diffeomorphism and with $\CP^2$-stable homeomorphism/diffeomorphism.
The $\CP^2$-stable classification is one of the most tractable~\cites{KPT-CP2,KT-CP2}.
It is easy to see that it differs markedly from the previously discussed classifications, as follows.

\begin{example}
  The $4$-manifolds $S^2 \times S^2$ and $S^2 \mathbin{\wt{\times}} S^2$ are smooth, simply connected, have equal Euler characteristic, and are $\CP^2$-stably diffeomorphic but are not stably homeomorphic and not homotopy equivalent. See \cref{subsection:row-2}.
\end{example}

The diffeomorphism classification, by contrast, is extremely difficult, and in all known cases varies drastically from the corresponding homeomorphism classifications.

\begin{example}
 There are infinitely many smooth, orientable, simply connected 4-manifolds that are all smoothly $s$-cobordant and homeomorphic to one another, but not diffeomorphic (see e.g.~\cites{GompfStip,akbulut-book}). Since exotic behaviour of this sort is not our primary focus, we only present the first pair of such manifolds discovered, due to Donaldson~\cite{Donaldson:1987-1}, in \cref{subsection:row-8}.  It follows that there is no smooth $s$-cobordism theorem in dimension $4$.
\end{example}



There are three implications in~\cref{fig:flowchart} for which we do not yet know whether the converses hold.

\begin{question}\label{q:he-not-she}
 Does there exist a pair of closed $4$-manifolds that are homotopy equivalent but not simple homotopy equivalent?
\end{question}

It would be interesting if such examples could be found which are (i) smoothable, (ii) orientable, (iii) topologically $h$-cobordant, or (iv) smoothly $h$-cobordant. The most well-known examples in odd dimensions of homotopy equivalent, but not simple homotopy equivalent, manifolds are lens spaces. The na\"ive construction of taking the products of homotopy equivalent lens spaces with $S^1$ does not work by the formula for Whitehead torsion~\eqref{it:lens-4}.

\begin{question}
 Is there a pair of $4$-manifolds that are (topologically) $s$-cobordant but not homeomorphic?
\end{question}

Note that a positive answer to this question would contradict the conjecture that all groups are good. For more details on this conjecture, see e.g.~\cite{DET-book-enigmata}. For more on the $s$-cobordism theorem in dimension four, see \cref{section:just-of-implications}.

\begin{question}
 Is there a pair of smooth $4$-manifolds that are smoothly $h$-cobordant but not smoothly $s$-cobordant?
\end{question}

\cref{intro-thm-2} provides topological examples of this phenomenon; since the construction uses the surgery sequence the examples are not obviously smoothable.

\cref{intro-example-CS-K-Ak} gives nonorientable examples for the following question, but in the orientable case this is open.
Note that smooth, simply connected $4$-manifolds that are homotopy equivalent are smoothly $h$-cobordant by Wall's theorem~\cite{Wall-on-simply-conn-4mflds}.

\begin{question}
 Is there a pair of smooth, orientable $4$-manifolds that are topologically but not smoothly $h$-cobordant?
\end{question}



As mentioned before, we have restricted ourselves throughout this paper to closed $4$-manifolds. However, interesting phenomena also arise for $4$-manifolds with nonempty boundary and for noncompact $4$-manifolds, e.g.\ the existence of corks~\cite{akbulut-fake-compact} and exotic smooth structures on $\R^4$~\cite{gompf-exotic} respectively. Other work in these directions include~\cites{akbulut-zeeman,taubes-uncountable,Boyer1,Boyer2,Vogel, Stong-4-manifolds-boundary,akbulut-ruberman,gompf-MCG,CP,CPP,orson-powell-mcg,orson-powell-spines}.

We hope that readers will be motivated by this article to answer the questions we could not, or to follow the paradigm of progressively imposing restrictions to discover new unanswered questions of their own.

\addtocontents{toc}{\SkipTocEntry}
\subsection*{Outline}
In \cref{sec:defs} we define the equivalence relations we consider. In~\cref{section:just-of-implications} we justify the implications shown in \cref{fig:flowchart}. \cref{Section:review-surgery} provides a brief review of the surgery exact sequence. In~\cref{section:counter-examples} we describe various constructions of $4$-manifolds and present a table summarising the properties of our examples.
\addtocontents{toc}{\SkipTocEntry}
\subsection*{Conventions}
We write $\Z/2 = \{0,1\}$ for the integers modulo 2, a group under addition, and $C_2 = \{\pm 1\}$ for the cyclic group of order 2, with multiplication as the group operation. The symbol $\simeq$ denotes homotopy equivalence, while $\simeq_s$ denotes simple homotopy equivalence. Depending on the context the symbol $\cong$ denotes either homeomorphism or diffeomorphism.
\addtocontents{toc}{\SkipTocEntry}
\subsection*{Acknowledgements}

The authors thank the Max Planck Institute for Mathematics in Bonn, where much of the research leading to this article occurred.
We also thank Jim Davis, Kent Orr, Patrick Orson, and Peter Teichner for several helpful discussions.
Finally we are very grateful to an anonymous referee for many helpful comments which helped us to improve the exposition.

DK was supported by the Deutsche Forschungsgemeinschaft under Germany's Excellence Strategy - GZ 2047/1, Projekt-ID 390685813.
MP was partially supported by EPSRC New Investigator grant EP/T028335/1 and EPSRC New Horizons grant EP/V04821X/1.

\tableofcontents

\section{Equivalence relations on 4-manifolds}
\label{section:definitions-of-relations}\label{sec:defs}

Recall that we implicitly assume throughout that $4$-manifolds are closed and connected. We assume that the reader is familiar with homotopy equivalence, homeomorphism, and diffeomorphism of manifolds, and so we shall not define them.
The classification of manifolds with respect to these three notions, and their comparison, is a central area of research.
For example, the Poincar\'e conjecture, which has occupied topologists for over a century, asks for each $n$ whether every homotopy equivalence from an $n$-manifold to the $n$-sphere is homotopic to a homeomorphism, or even to a diffeomorphism.\footnote{The Poincar\'{e} conjecture is true in the topological category for all $n$, due to Perelman~\cites{Perelman:2002-1,Perelman:2003-1,Perelman:2003-2}, Freedman~\cite{F}, and Newman~\cite{Newman}. It is true in the PL category for all $n \neq 4$ due to Perelman [loc.~cit.], Smale~\cites{smale-poincare-bulletin,Smale-h-cob-thm}, Stallings~\cite{Stallings-BAMS}, and Zeeman~\cite{zeeman-poincare}. In the smooth category it is known to hold in dimensions $1,2,3,5,6,12,56,61$, due to Perelman~[loc.~cit.], Kervaire-Milnor~\cite{Kervaire-Milnor:1963-1}, Isaksen~\cite{Isaksen}, and Wang-Xu~\cite{Wang-Xu}; is false in all odd dimensions other than $1,3,5,61$; and is false in all even dimensions $8 \leq n \leq 200$ other than $12$, $56$, $142$, $166$, $176$, $188$. At the time of writing the smooth version is open in dimensions $4$, $142$, $166$, $176$, $188$, and for infinitely many even dimensions~$n > 200$; see \cites{Behrens-Hill-Hopkins-Mahowald,Isaksen-Wang-Xu-PNAS} for the published state of the art.}   In dimensions at least five, surgery theory provides a concrete, effective framework within which one can try to improve a classification of manifolds up to homotopy equivalence to a classification up to homeomorphism or diffeomorphism. The programme can be applied to topological 4-manifolds under a restriction on the fundamental group; see \cref{sec:surgery} for an overview.

Next we discuss the various notions of stable equivalence.

\begin{definition}\label{defn:stable-homeo}
The $4$-manifolds $M$ and $N$ are said to be \emph{stably homeomorphic} if there  are integers~$s,t$ such that $M \# \bighash{s} (S^2 \times S^2)$ and $N \#  \bighash{t} (S^2 \times S^2)$ are homeomorphic.
 They are said to be \emph{$\CP^2$-stably homeomorphic} if there are integers $s,t$ such that $M \# \bighash{s} \CP^2$ and $N \# \bighash{t} \CP^2$ are homeomorphic, for some choices of connected sum.
\end{definition}

\begin{definition}\label{defn:stab-diff}
  The smooth $4$-manifolds $M$ and $N$ are said to be \emph{stably diffeomorphic} if there are integers $s,t$ such that $M \# \bighash{s} (S^2 \times S^2)$ and $N \# \bighash{t} (S^2 \times S^2)$ are diffeomorphic. They are said to be \emph{$\CP^2$-stably diffeomorphic} if there are integers $s,t$ such that $M  \# \bighash{s} \CP^2$ and $N \# \bighash{t} \CP^2$ are diffeomorphic, for some choices of connected sum.
\end{definition}

Note that $S^2 \times S^2$ admits an orientation reversing self-diffeomorphism, so there is essentially only one choice of connected sum. On the other hand $\CP^2$ does not admit any such diffeomorphism (nor homeomorphism), so for oriented manifolds there are two possible connected sums up to diffeomorphism/homeomorphism, usually denoted $M \# \CP^2$ and $M \# \ol{\CP^2}$. Therefore, the definitions above say that $M$ and $N$ are $\CP^2$-stably diffeomorphic (resp.\ homeomorphic) if there are integers $s_1,s_2,t_1,t_2$ such that $M  \# \bighash{s_1} \CP^2 \# \bighash{s_2} \ol{\CP^2}$ and $N  \# \bighash{t_1} \CP^2 \# \bighash{t_2} \ol{\CP^2}$ are diffeomorphic (resp.\ homeomorphic). Note that for nonorientable manifolds, there is a unique connected sum~$N \# \CP^2$. We remark that some authors require $s=t$ in the definition of stable homeomorphism and diffeomorphism. Manifolds $M$ and $N$ which are stably homeomorphic or diffeomorphic have $s=t$ in our definition exactly when $\chi(M)=\chi(N)$.

We emphasise that ``stably'' refers by default to connected sum with copies of $S^2 \times S^2$, and only ``$\CP^2$-stably'' refers to connected sum with copies of $\CP^2$.

To motivate \cref{defn:stable-homeo,defn:stab-diff}, consider an alternative strategy to classify manifolds, based on Kreck's modified surgery~\cite{surgeryandduality} and realised by Hambleton-Kreck for e.g.\ 4-manifolds with finite cyclic fundamental group in~\cites{Hambleton-Kreck:1988-1,HK88b,Hambleton-Kreck-93,HK93b}: first classify manifolds up to stable homeomorphism, and then investigate the homeomorphism types within each stable class. In the latter step one attempts to prove that $S^2\times S^2$ summands can be cancelled, and this is what Hambleton and Kreck achieved for finite cyclic fundamental groups and also for more general finite groups under some additional hypotheses.
This strategy in principle also applies to diffeomorphism classifications, but the cancellation step is much harder.  Similarly, another approach to classification is to first classify manifolds up to $\CP^2$-stable equivalence, and then attempt to blow down extraneous $\CP^2$ summands.

Next we discuss $h$-cobordisms, simple homotopy equivalences, and $s$-cobordisms.

\begin{definition}
The $4$-manifolds $M$ and $N$ are \emph{topologically $h$-cobordant} if there is a $5$-dimensional compact topological cobordism $(W;M,N)$ where the inclusion maps $M \hookrightarrow W$ and $N \hookrightarrow W$ are homotopy equivalences. The manifold $W$ is called an \emph{$h$-cobordism}. If $M$ and $N$ are smooth, they are \emph{smoothly $h$-cobordant} if they cobound a smooth $h$-cobordism.
\end{definition}

Associated with a homotopy equivalence $f \colon X \to Y$ between CW complexes $X$ and $Y$ is an algebraic invariant called the \emph{Whitehead torsion} $\tau(f) \in \Wh(\pi_1(X))$, with values in the \emph{Whitehead group} of $\pi_1(X)$, which we define next. Let $\GL(\Z[\pi_1(X)])$ be the stable general linear group, and let $\E(\Z[\pi_1(X)])$ be the subgroup of elementary matrices, i.e.\ consisting of products of the matrices that produce row and column operations.  By definition \[
\K_1(\Z[\pi_1(X)]) := \GL(\Z[\pi_1(X)])/\E(\Z[\pi_1(X)])
\]
and $\Wh(\pi_1(X)) := \K_1(\Z[\pi_1(X)])/\{\pm (g) \mid g \in \pi_1(X)\}$. For example $\Wh(\{e\}) = 0$, essentially because of the Euclidean algorithm.  See~\cite{Cohen} for an accessible introduction to simple homotopy theory, including more examples of Whitehead groups and the definition of Whitehead torsion.

\begin{definition}
  A homotopy equivalence $f \colon X \to Y$ between CW complexes $X$ and $Y$ is a \emph{simple homotopy equivalence} if its Whitehead torsion $\tau(f)$ vanishes.
\end{definition}

By Chapman's theorem~\cite{chapman} the Whitehead torsion $\tau(f)$ only depends on the homeomorphism type of $X$ and $Y$. Hence Whitehead torsion is well-defined for homotopy equivalences between manifolds which are homeomorphic to CW complexes, e.g. smooth manifolds or closed manifolds of dimension $\neq 4$. It is an open question whether every topological $4$-manifold is homeomorphic to a $4$-dimensional CW complex.
However, we can define the notion of simple homotopy equivalence of topological manifolds as follows.
Embed $M$ in high-dimensional Euclidean space. By \cite{Kirby-Siebenmann:1977-1}*{Essay III, Section 4}, there is a normal disc bundle $D(M)$ admitting a triangulation. The inclusion map $z_M \colon M \to D(M)$ of the $0$-section 
is a homotopy equivalence.
Let $z_M^{-1}$ denote the homotopy inverse of $z_M$.

\begin{definition}
 We say that a homotopy equivalence $f \colon M \to N$ between topological manifolds (not necessarily of the same dimension) is a \emph{simple homotopy equivalence} if the composition $z_N \circ f \circ z_M^{-1} \colon D(M) \to D(N)$ is a simple homotopy equivalence.
\end{definition}

The Whitehead torsion $\tau(W;M)$ of an $h$-cobordism $(W;M,N)$ is by definition the Whitehead torsion of the inclusion map $M \hookrightarrow W$. This also coincides with the Whitehead torsion of the relative chain complex $C_*(\wt{W},\wt{M})$, where $\wt{W}$ and $\wt{M}$ are the universal covers.

\begin{definition}
The $4$-manifolds $M$ and $N$ are \emph{topologically $s$-cobordant} if they cobound a topological $h$-cobordism $W$ with trivial Whitehead torsion. The manifold $W$ is called an \emph{$s$-cobordism}. If in addition $M$, $N$, and $W$ are smooth, then $M$ and $N$ are \emph{smoothly $s$-cobordant} and $W$ is called a \emph{smooth $s$-cobordism}.
\end{definition}

An $h$- or $s$-cobordism approximates a product, in the eyes of homotopy equivalence and simple homotopy equivalence respectively.
One of the most spectacular results of the 20th century was Smale's $h$-cobordism theorem~\cites{smale-poincare-bulletin,Smale-h-cob-thm}, which states that smooth, simply connected $h$-cobordisms between $n$-manifolds with $n\geq 5$ are indeed homeomorphic to products. This was later extended to other categories and to the case of $s$-cobordisms~\cites{Smale:1962-1,Barden,Mazur:1963-1,Stallings-Tata,Kirby-Siebenmann:1977-1}. Consequences include the high-dimensional Poincar\'{e} conjecture in the piecewise-linear category in dimension at least five.

In dimension four, the celebrated work of Freedman and Quinn~\cites{F,FQ} includes an $s$-cobordism theorem, with a restriction on fundamental groups. This is the principal method for establishing the existence of a homeomorphism between $4$-manifolds.   We state the result in the next section as \cref{thm:scob} and we outline the proof.

\section{Justification of implications}\label{section:just-of-implications}
The implications given in green in \cref{fig:flowchart} are immediate from the definitions. Now we justify the other implications.

\begin{proposition}
Stably homeomorphic $4$-manifolds are $\CP^2$-stably homeomorphic. Stably diffeomorphic $4$-manifolds are $\CP^2$-stably diffeomorphic.
\end{proposition}

\begin{proof}
Both statements follow from the diffeomorphism
\[
(S^2 \times S^2) \# \CP^2 \cong \CP^2 \# \ol{\CP^2} \# \CP^2.\qedhere
\]
\end{proof}

To establish the relationship between homotopy equivalence and $\CP^2$-stable homeomorphism, we will use the following theorem of Kreck. Recall that an identification of the fundamental group $\pi_1(M)$ of a $4$-manifold $M$ with a group $\pi$ determines a map $c_M \colon M \to B\pi$, up to homotopy, classifying the universal cover, where $B\pi \simeq K(\pi,1)$ is the classifying space.

\begin{theorem}[Kreck~\cite{surgeryandduality}, see also~\cite{KPT}*{Theorem~1.2}]\label{thm:kreck-cp2}
\leavevmode
\begin{enumerate}[(i)]
\item\label{item-kreck-1}   Two closed, smooth $4$-manifolds $M$ and $N$ with fundamental group isomorphic to $\pi$ and orientation character $w \colon \pi \to C_2$  are $\CP^2$-stably diffeomorphic if and only if $c_{M*}[M] = c_{N*}[N]  \in H_4(\pi;\Z^w)/\pm \Aut(\pi,w)$.
\item\label{item-kreck-2} Two closed, topological $4$-manifolds $M$ and $N$ with fundamental group isomorphic to $\pi$ and orientation character $w \colon \pi \to C_2$  are $\CP^2$-stably homeomorphic if and only if their Kirby-Siebenmann invariants in $\Z/2$ agree and $c_{M*}[M] = c_{N*}[N]  \in H_4(\pi;\Z^w)/\pm \Aut(\pi,w)$.
\end{enumerate}
\end{theorem}

Here $\Aut(\pi,w)$ denotes the set of automorphisms of $\pi$ compatible with the map $w$. We have to factor out by the action of $\pm \Aut(\pi,w)$ in order to account for the choice of identifications $\pi_1(M) \cong \pi$ and $\pi_1(N) \cong \pi$, and for the choice of (twisted) fundamental classes in $H_4(-;\Z^w)$.

The \emph{Kirby-Siebenmann invariant} $\ks(M)\in\Z/2$ of a $4$-manifold $M$ is by definition the unique obstruction for the stable tangent microbundle of $M$ to admit a lift to a piecewise linear bundle. See~\citelist{\cite{Kirby-Siebenmann:1977-1}*{p.\ 318}\cite{FQ}*{Section~10.2B}\cite{guide}*{Section~8.2}} for further details on the definition. In general it will suffice for us to know that the Kirby-Siebenmann invariant satisfies strong additivity properties, in particular under gluing and connected sum, and that it vanishes for a $4$-manifold $M$ if and only if $M\times \R$ admits a smooth structure, if and only if $M \#  \bighash{k} S^2\times S^2$ admits a smooth structure for some $k$.

\begin{theorem}\label{thm:apply-kreck}~
\begin{enumerate}
\item\label{item:theorem-on-equiv-relations-1}   Homotopy equivalent 4-manifolds with equal Kirby-Siebenmann invariants are $\CP^2$-stably homeomorphic.
\item\label{item:theorem-on-equiv-relations-3} Smooth $4$-manifolds that are $\CP^2$-stably homeomorphic are also $\CP^2$-stably diffeomorphic.
\end{enumerate}
\end{theorem}

\begin{proof}
For the first implication, let $M$ and $N$ be homotopy equivalent $4$-manifolds with equal Kirby-Siebenmann invariants. Fix a map $c_N\colon N\to B\pi$ as mentioned above the statement of \cref{thm:kreck-cp2} and let $f\colon M\to N$ be the claimed homotopy equivalence. Then define $c_M:=c_N\circ f$. This ensures that $c_{M*}[M] = c_{N*}[N]  \in H_4(\pi;\Z^w)/\pm \Aut(\pi,w)$, and so by~\cref{thm:kreck-cp2}\,\eqref{item-kreck-2}, we see that $M$ and $N$ are $\CP^2$-stably homeomorphic.

For the second implication, let $M$ and $N$ be $\CP^2$-stably homeomorphic and smooth. By~\cref{thm:kreck-cp2}\,\eqref{item-kreck-2} we see that $c_{M*}[M] = c_{N*}[N]  \in H_4(\pi;\Z^w)/\pm \Aut(\pi,w)$. Then apply~\cref{thm:kreck-cp2}\,\eqref{item-kreck-1} to see that $M$ and $N$ are $\CP^2$-stably diffeomorphic.
\end{proof}

Next we show the relationship between $h$-cobordism and stable diffeomorphism. The case of simply connected $4$-manifolds was addressed by Wall in \cite{Wall-on-simply-conn-4mflds}*{Theorem~3}. A similar argument also applies in the general setting as explained by Lawson~\cite{Lawson-decomposing-5-mnfs}*{Proposition}. We sketch the proof.

\begin{theorem}[\citelist{\cite{Wall-on-simply-conn-4mflds}*{Theorem~3}\cite{Lawson-decomposing-5-mnfs}*{Proposition}}]\label{thm:h-cob-implies-stably-diff-homeo}
Smoothly $h$-cobordant 4-manifolds are stably diffeomorphic. Similarly, topologically $h$-cobordant 4-manifolds are stably homeomorphic.
\end{theorem}

\begin{proof}
The proof is the same in both cases, by using the fact that $5$-dimensional topological cobordisms $(W;M,N)$ admit handle decompositions, i.e.\ $W$ can be built by attaching $5$-dimensional handles to $M\times [0,1]$ along topological embeddings of the attaching regions~\citelist{\cite{Quinn-annulus}*{Theorem~2.3.1}\cite{FQ}*{Theorem~9.1}}. In the case of a smooth $h$-cobordism, we get a smooth handle decomposition relative to $M$ by Morse theory.

In either case, we can perform handle trading to ensure that the handle decomposition has only $2$- and $3$-handles since the boundary inclusions $M\hookrightarrow W$ and $N\hookrightarrow W$ are $1$-connected. Consider the middle level $M_{1/2}$ of the cobordism, obtained after attaching $2$-handles to $M$. Since $W$ is an $h$-cobordism, the $2$-handles are attached along trivial circles, and so
\[
M_{1/2}\cong M \# \bighash{t_1} S^2\times S^2  \# \bighash{t_2} S^2 \mathbin{\wt{\times}} S^2,\]
where $t_1+t_2$ is the number of $2$-handles. Here $S^2\mathbin{\wt{\times}} S^2$ is the twisted $S^2$-bundle over $S^2$.

If $t_2=0$ we are done. In case $t_2\neq 0$, then there is an embedded $2$-sphere in $M_{1/2}$ with odd framing of its normal bundle. Via the homotopy equivalence $W\to M$, we see there is a map of a sphere to $M$ with odd framing of its normal bundle, implying that the universal cover of $M$ is non-spin. In this case we have that $M\# (S^2\mathbin{\wt{\times}} S^2)\cong M\# (S^2\times S^2)$ (see e.g.\ \cite{GompfStip}*{Exercise~5.2.6(b)}), and so we may assume $t_2=0$ again.

We have now argued that $M_{1/2}$ is a stabilisation of $M$. Applying the same argument to the upside down handlebody, we see further that $M_{1/2}$ is a stabilisation of $N$.
Thus $M$ and $N$ are stably homeomorphic or diffeomorphic, depending on whether the handle decomposition was smooth or merely topological to begin with.
\end{proof}

Next we discuss the $s$-cobordism theorem in dimension four. Below, we say a group is \emph{good} if it satisfies the $\pi_1$-null disc property~\cite{Freedman-Teichner:1995-1} (see also~\cite{Freedman-book-goodgroups}). We do not repeat the definition here.
In practice, it generally suffices to know that virtually solvable groups, and more generally groups of subexponential growth are good, and the class of good groups is closed under taking subgroups, quotients, extensions, and colimits~\cites{Freedman-Teichner:1995-1, Krushkal-Quinn:2000-1}.

\begin{theorem}[$s$-cobordism theorem]
	\label{thm:scob}
	Let $M$ be a topological $4$-manifold with $\pi:=\pi_1(M)$ a good group.
	\begin{enumerate}
		\item\label{it:s-cob1} Let $(W;M,M')$ be an $h$-cobordism over $M$. Then $W$ is trivial over $M$, i.e.\ $W\cong M\times[0,1]$, via a homeomorphism restricting to the identity on $M$, if and only if its Whitehead torsion $\tau(W;M)\in\Wh(\pi)$ vanishes.
		\item\label{it:s-cob2} For any $\varsigma\in\Wh(\pi)$ there exists an $h$-cobordism $(W;M,M')$ with $\tau(W;M)=\varsigma$.
		\item\label{it:s-cob3} The function assigning to an $h$-cobordism $(W;M,M')$ its Whitehead torsion $\tau(W;M)$ yields a bijection from the homeomorphism classes relative to
		$M$ of $h$-cobordisms over $M$ to the Whitehead group $\Wh(\pi)$.
	\end{enumerate}
\end{theorem}

\begin{remark}
It was asserted in \cite{Rourke-Sanderson}*{p.\ 90} that \cref{thm:scob}\,\eqref{it:s-cob2} holds in dimension 4 in the piecewise-linear category, and without the assumption on $\pi_1(M)$. But the proof there does not take into account the need for geometrically dual spheres to control the fundamental group of $M'$.
\end{remark}

\begin{proof}[Proof of \cref{thm:scob}]
	The statement~\eqref{it:s-cob1} is \cite{FQ}*{Theorem~7.1A}, and relies on Freedman's disc embedding theorem~\cite{F} (see also~\cites{FQ,Freedman-notes}). We prove \eqref{it:s-cob2}, which follows the high dimensional argument from \cite{Milnor-whitehead-torsion}*{Theorem~11.1}.

Let $A\in \GL_k(\Z\pi)$ represent $\varsigma$. Attach $k$ trivial $2$-handles to $M\times \{1\}\subseteq M\times [0,1]$. This yields a bordism $W'$ from $M$ to $M \# \bighash{k} S^2\times S^2$, where we assume that the ascending spheres of the $2$-handles are $S^2\times\{s_0\}$ in each copy of $S^2\times S^2$. For every $1\leq i\leq k$, we can realize the element $(0,(0,A_{ij})_{j=1}^k)\in \pi_2(M)\oplus\bigoplus_{j=1}^k(\Z\pi)^2\cong \pi_2(M \# \bighash{k} S^2\times S^2)$ by an embedded framed sphere by tubing together parallel copies of the embedded framed spheres $\{s_0\}\times S^2$ in $S^2\times S^2$. Let $\{f_i\}_{i=1}^k$ be the resulting collection of spheres.  These spheres admit pairwise disjoint algebraically dual spheres $\{g_i\}_{i=1}^k$, obtained by tubing together parallel copies of the embedded framed spheres $S^2 \times \{s_0\}$ in order to realise the rows of the matrix $(\ol{A}^T)^{-1}$.  By the sphere embedding theorem~\citelist{\cite{FQ}*{Theorem~5.1B}\cite{PRT20}*{Theorem~B}} there is a collection $\{\ol{f}_i\}_{i=1}^k$ of topologically flat embedded spheres with $\ol{f}_i$ homotopic to $f_i$ and such that the collection $\{\ol{f}_i\}_{i=1}^k$ admits a geometrically dual collection of (immersed) spheres $\{\ol{g}_i\}_{i=1}^k$.

Attach $3$-handles to $W'$ along neighbourhoods of the spheres $\{\ol{f}_i\}_{i=1}^k$ to obtain a bordism $W$ from $M$ to $M'$. Since we attached the handles of index 2 along trivial circles, $\pi_1(M)\cong \pi_1(W)$.  Since the $\{\ol{f}_i\}$ have geometrically dual spheres, surgery along them does not change the fundamental group and so we have $\pi_1(M')\cong \pi_1(W)$.
The handle chain complex $C_*(\wt{W},\wt{N})$ is \[0 \to C_3(\wt{W},\wt{M}) \cong \oplus^k \Z\pi \xrightarrow{A} C_2(\wt{W},\wt{M}) \cong \oplus^k \Z\pi \to 0.\]
Since $A$ is invertible, $H_*(\wt{W},\wt{M})=0$, and then by duality $H_*(\wt{W},\wt{M'})=0$ too. Therefore $W$ is an $h$-cobordism.
The torsion can be read off from the handle chain complex as $\tau(W;M) = [A] = \varsigma$.

Finally, \eqref{it:s-cob3} is a consequence of~\eqref{it:s-cob1} and~\eqref{it:s-cob2}. Surjectivity is immediate from~\eqref{it:s-cob2}. The proof of injectivity follows~\cite{Milnor-whitehead-torsion}*{Theorem~11.3} and uses~\eqref{it:s-cob1} and~\eqref{it:s-cob2}. Let $(W;M,M')$ and $(W';M,M'')$ be $h$-cobordisms over $M$ with torsion $\varsigma$. By~\eqref{it:s-cob2}, there is a 4-manifold $N$ and an $h$-cobordism $(W'';M',N)$ with torsion $-\varsigma$. By the additivity of Whitehead torsion~\cite{lueck-surgery-intro}, $(W\cup_{M'}W'',M,N)$ is an $s$-cobordism. Since $\pi$ is good, $W\cup_{M'}W''$ is homeomorphic to $M\times [0,1]$ relative to $M$ by~\eqref{it:s-cob1}.  In particular $N$ is homeomorphic to $M$. We can thus form $(W'' \cup_{M} W';M',M'')$, which again is an $s$-cobordism and is thus homeomorphic to $M'\times [0,1]$, relative to $M'$ by \eqref{it:s-cob1}. We obtain a homeomorphism
\[W\cong W\cup_{M'}(M'\times[0,1])\cong W\cup_{M'}W''\cup_{M}W'\cong (M\times [0,1])\cup_{M'}W'\cong W',\]
relative to $M$, as claimed.
\end{proof}

\begin{remark}
Every $5$-dimensional $s$-cobordism $(W;M,N)$, with no restriction on fundamental groups, becomes homeomorphic to a product, relative to $M$, after sufficiently many connected sums with $S^2\times S^2\times [0,1]$ along arcs joining $M$ and $N$~\cite{quinn-1983}*{Theorem~1.1}. This gives another proof of the special case of \cref{thm:h-cob-implies-stably-diff-homeo}  that $s$-cobordant $4$-manifolds are stably diffeomorphic/homeomorphic.
\end{remark}

The last remaining implication in \cref{fig:flowchart} is given by the following result of Gompf.

\begin{theorem}[\cite{Gompf-stable}]\label{thm:gompf-stable}
Smooth, orientable $4$-manifolds that are stably homeomorphic are also stably diffeomorphic.
\end{theorem}

In the same paper Gompf also showed that smoothings of a nonorientable $4$-manifold become diffeomorphic after connected sum with sufficiently many copies of $S^2\mathbin{\wt{\times}}S^2$.

\section{Review of surgery exact sequences}\label{Section:review-surgery}\label{sec:surgery}

In the next section we will appeal on several occasions to surgery exact sequences. We refer to \cite{Freedman-book-surgery} for an account of the 4-dimensional case, and e.g.\  \cite{Wall-surgery-book} and \cite{CLM} for detailed treatments of general surgery theory in dimensions at least $5$.

The surgery exact sequences are centred on the \emph{structure sets.}  For the remainder of this section, let $M$ be a closed, connected, topological 4-manifold.

\begin{definition}
The \emph{homotopy structure set} of $M$, denoted $\mathcal{S}^h(M)$, is by definition the set of pairs $(N,f \colon N \xrightarrow{\simeq} M)$, where $N$ is a closed topological 4-manifold and $f$ is a homotopy equivalence, considered up to $h$-cobordism over $M$. That is, $[(N,f)] =[(N,f')] \in \mathcal{S}^h(M)$ if and only if there is an $h$-cobordism $(W;N,N')$, with inclusion maps $i \colon N \to W$ and $i' \colon N' \to W$, together with a map $F \colon W \to M$ such that $F \circ i = f$ and $F \circ i' =f'$.
\end{definition}

Note that $F$ is necessarily also a homotopy equivalence in the definition above.

\begin{definition}
The \emph{simple structure set} of $M$, denoted $\mathcal{S}^s(M)$, is by definition the set of pairs $(N,f \colon N \xrightarrow{\simeq_s} M)$, where $N$ is a closed topological 4-manifold and $f$ is a simple homotopy equivalence, considered up to $s$-cobordism over $M$. That is, $[(N,f)] =[(N',f')] \in \mathcal{S}^s(M)$ if and only if there is an $s$-cobordism $(W;N,N')$, with inclusion maps $i \colon N \to W$ and $i' \colon N' \to W$, together with a map $F \colon W \to M$ such that $F \circ i = f$ and $F \circ i' =f'$.
\end{definition}

Note that $F$ is necessarily a simple homotopy equivalence.
Suppose that $\pi_1(M)$ is a good group. Then every $s$-cobordism is a product, and we can alternatively describe the equivalence relation in the definition of the simple structure set without reference to $s$-cobordisms, by instead requiring a homeomorphism $G \colon N \to N'$ such that there is a homotopy $f' \circ G \sim f \colon N\to M$.

Continuing with the assumption that $\pi_1(M)$ is good, it follows that one approach to the classification of manifolds simple homotopy equivalent to~$M$, up to homeomorphism, is to first compute the simple structure set of $M$, and then to compute the set of orbits of the post-composition action on it by the group $\hAut^s(M)$ of homotopy classes of simple self-homotopy equivalences of $M$. The set of orbits is then the set of homeomorphism classes $\mathcal{M}(M)$ of manifolds simple homotopy equivalent to $M$.

Similarly, the set $\mathcal{M}^h(M)$ of $h$-cobordism classes of closed 4-manifolds homotopy equivalent to $M$ is in bijective correspondence with the orbits of $\mathcal{S}^h(M)$ under the action of the group $\hAut(M)$ of homotopy classes of self-homotopy equivalences of~$M$.

When $\pi_1(M)$ is good, the surgery sequences are exact sequences of abelian groups, whose underlying sets are given as follows; the group structures arise via the theory of spectra and are hard to define geometrically~\cite{nicas-memoir}*{Chapter~5}. The identity element of the structure sets is given by the identity map $M\to M$.
We will explain the terms other than the structure sets after stating the sequences. Let $\pi := \pi_1(M)$ and let $w \colon \pi \to C_2$ be the orientation character of~$M$. Assume that $\pi$ is a good group.  For the homotopy structure set, we have an exact sequence:
\[
\begin{tikzcd}
\mathcal{N}(M \times [0,1],M \times \{0,1\})\arrow[r,"\sigma"]	&L_5^h(\Z\pi,w)\arrow[r,"W"]	&\mathcal{S}^h(M)\arrow[r,"\eta"]	&\mathcal{N}(M)\arrow[r,"\sigma"]	&L_4^h(\Z\pi,w),
\end{tikzcd}
\]
while for the simple structure set we have an exact sequence:
\[
\begin{tikzcd}
\mathcal{N}(M \times [0,1],M \times \{0,1\})\arrow[r,"\sigma"]	&L_5^s(\Z\pi,w)\arrow[r,"W"]	&\mathcal{S}^s(M)\arrow[r,"\eta"]	&\mathcal{N}(M)\arrow[r,"\sigma"]	&L_4^s(\Z\pi,w),
\end{tikzcd}
\]

The \emph{degree one normal maps}, the terms involving $\mathcal{N}$, are independent of the $h$ and $s$ decorations.  For $(X,\partial X)$ equal to either $(M,\emptyset)$ or $(M \times [0,1],M \times \{0,1\})$, the set $\mathcal{N}(X,\partial X)$ consists of the set of manifolds $(N,\partial N)$ with a degree one map $N \to X$ that restricts to a homeomorphism on $\partial N \to \partial X$, together with some normal bundle data that we will not define here, up to an analogous notion of degree one normal bordism. Details can be found in the references provided at the start of this section. It will suffice for us to know that $\cN(M)\cong [M, \G/\TOP]$ and $\cN(M\times [0,1],M\times \{0,1\})\cong [(M\times [0,1],M\times \{0,1\}), (G/\TOP, *)]$. For us the only relevant property of the space $\G/\TOP$ will be the existence of a $5$-connected map $\G/\TOP \to K(\Z,4) \times K(\Z/2,2)$~\citelist{\cite{Madsen-Milgram}\cite{Kirby-Taylor}*{p.\ 397}}, so in particular
\begin{equation}\label{eq:cohomology-normalinvariants}
\cN(M)\cong H^4(M;\Z) \oplus H^2(M;\Z/2).
\end{equation}
Similarly, for $Y$ a closed $3$-manifold, we have $[Y, \G/\TOP]\cong H^2(Y;\Z/2)$. As before the identity element in the set of normal maps is given by the identity map.

The \emph{$L$-groups} $L_5^h(\Z\pi,w)$, $L_5^s(\Z\pi,w)$, $L_4^h(\Z\pi,w)$, and $L_4^s(\Z\pi,w)$ have purely algebraic definitions, in terms of $\pi$, $w$, and the decoration $h$ or $s$. We will also not go into the details of the definitions here, but will give a brief overview.

Roughly speaking, $L_4^h(\Z\pi,w)$ is defined in terms of nonsingular, sesquilinear, hermitian forms on free $\Z\pi$-modules. For $L_4^s(\Z\pi,w)$ the $\Z\pi$-modules must be based, and certain isomorphisms between based modules are required to be simple, meaning that they represent the trivial element of the Whitehead group $\Wh(\pi)$. The identity element in the $L_4$ groups corresponds to the \emph{hyperbolic form}, i.e.\  $\bigoplus^k \bsm 0 & 1\\ 1 & 0\esm$ on $\oplus^{2k} \Z\pi$ for some $k$.  There is also a version $L_4^p(\Z\pi,w)$, where the underlying $\Z\pi$-modules are only required to be projective.

Elements of the $L_5$ groups consist of a hyperbolic form on a free $\Z\pi$-module equipped with a choice of a half rank summand of the base module, called a \emph{lagrangian}, on which the form vanishes. For the $s$ decorations,  we need the module to be based and a certain short exact sequence related to the lagrangian to have trivial Whitehead torsion. The identity element of the $L_5$-groups consists of a hyperbolic form where the lagrangian is standard.
When the orientation character $w$ is trivial, we often suppress $w$ from the notation of $L$-groups.

The homomorphisms $\sigma$ in the surgery sequences from degree one normal maps to the $L$-groups are called the \emph{surgery obstruction maps}. Given a degree one normal map $f\colon N\to M$, performing surgery on circles produces a map $f'$, still with target $M$, that induces an isomorphism on fundamental groups. The element $\sigma(f)$ is given by the kernel of the map induced by $f'$ on second homotopy groups, called the \emph{surgery kernel}. Exactness at $\cN(M)$ requires that $\pi_1(M)$ is good and relies on \cites{F,FQ}.

Note that $L_4(\Z)\cong \Z$, given by the signature of the form divided by $8$. Hence $\sigma\colon \cN(S^4)\to L_4(\Z)$ sends $[N,f]$ to $\sign(N)/8$. Using the naturality of the surgery exact sequence and that $L_4(\Z\pi)$ contains $L_4(\Z)$ as a direct summand, we see that the summand $H^4(M;\Z)\cong \Z$ in $\cN(M)$ is detected by the signature difference $[N,f]\mapsto (\sigma(N)-\sigma(M))/8$  and maps injectively to $L_4(\Z\pi)$. The preceding argument applies to $L_4(\Z\pi)$ with both $h$ and $s$ decorations.

The maps marked $W$ in the surgery sequences are given by the \emph{Wall realisation} actions of the $L_5$ groups on the structure sets, which we sketch next. Let $h\colon N\to M$ be a (simple) homotopy equivalence and let $\alpha$ be an element of the relevant $L_5$ group. Stabilising gives a map $N \# \bighash{k} S^2\times S^2 \to M$, for some $k$, whose surgery kernel gives a hyperbolic form. We then represent the generators of the lagrangian in $\alpha$ by framed, disjointly embedded $2$-spheres in $N \# \bighash{k} S^2\times S^2$, on which we perform surgery. The resulting $4$-manifold $N'$ comes equipped with a (simple) homotopy equivalence $h'\colon N'\to M$ and $(N',h')$ is by definition the element $\alpha\cdot h$ of the structure set. By construction $N$ and $N'$ are stably homeomorphic. It is highly nontrivial to represent the lagrangian by disjointly embedded spheres, and requires the work of \cites{F,FQ} and the restriction to good fundamental groups.

\section{Counterexamples}\label{section:counter-examples}

\cref{fig:flowchart-counterexamples} shows what we know about the converses of the implications in \cref{fig:flowchart}. In this section we collect the counterexamples indicated in \cref{fig:flowchart-counterexamples}, explaining their construction and properties. The properties are also collected in~\cref{table} at the end of the paper.

\begin{figure}[htb]
\includegraphics[width=15cm]{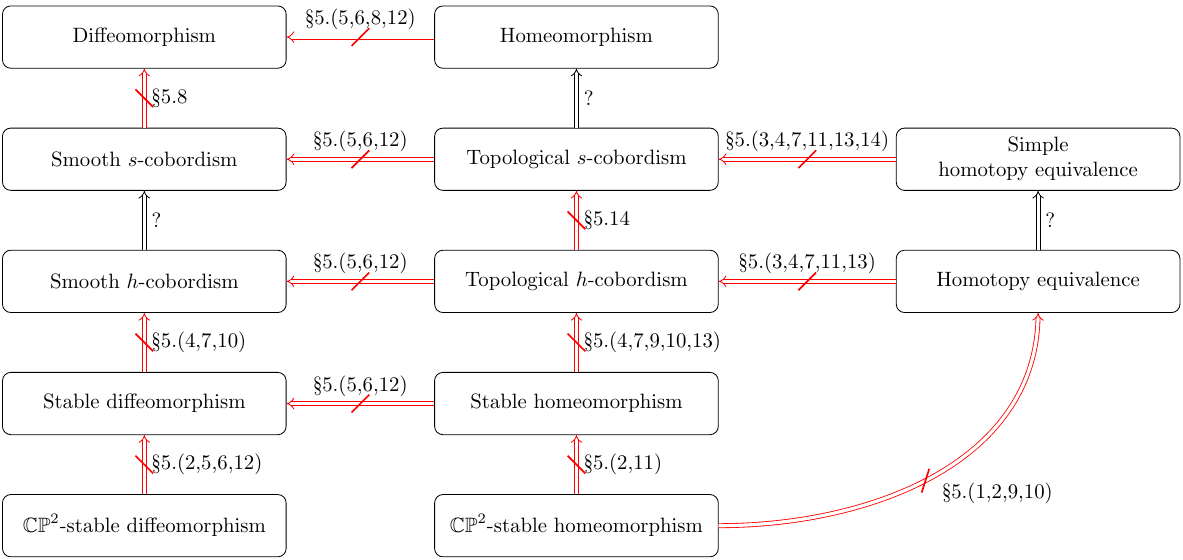}
\caption{What is known about the converses of the implications in \cref{fig:flowchart}. The symbol  \red{$\nRightarrow$} denotes the cases where we know an implication does not hold, indicating which subsections contain corresponding counterexamples. The $\xRightarrow{?}$ symbol denotes the three cases where it is unknown whether an implication holds.
}
\label{fig:flowchart-counterexamples}
\end{figure}

\subsection{\texorpdfstring{$S^2 \times S^2$ and $S^4$}{S2xS2 and S4}}\label{subsection:row-1}

\begin{itemize}
  \item Both manifolds are smooth, orientable, and simply connected.
\item As $S^2\times S^2 \cong S^4\# (S^2\times S^2)$, they are stably diffeomorphic, and therefore stably homeomorphic and $\CP^2$-stably diffeomorphic.
\item $\chi(S^2\times S^2)=4\neq 2=\chi(S^4)$. Therefore, they are neither (simple) homotopy equivalent, $s$- nor $h$-cobordant (in either category), homeomorphic, nor  diffeomorphic.
\end{itemize}

This straightforward example shows that to meaningfully ask for homotopically inequivalent 4-manifolds that are stably diffeomorphic, one should also require that the Euler characteristics coincide. Equivalently, one should  require that the number of copies of $S^2 \times S^2$ added is the same for both manifolds, that is $s=t$ in \cref{defn:stab-diff}.

\subsection{\texorpdfstring{$S^2 \times S^2$ and $S^2 \mathbin{\wt{\times}} S^2$}{S2xS2 and S2x~S2}}\label{subsection:row-2}

As before $S^2 \mathbin{\wt{\times}} S^2$ denotes the nontrivial $S^2$-bundle over $S^2$. It can be constructed by gluing two copies of $D^2 \times S^2$ together using the Gluck twist on their common boundary $S^1 \times S^2$.  Alternatively, recall that oriented $3$-plane bundles over $S^2$ are classified up to isomorphism by homotopy classes of maps $[S^2,BSO(3)] \cong [S^1,SO(3)]$, and there are two such homotopy classes. The nontrivial map gives a $3$-plane bundle whose sphere bundle is $S^2 \wt{\times} S^2$.

\begin{itemize}
  \item Both $S^2\times S^2$ and $S^2\mathbin{\wt{\times}} S^2$ are smooth, orientable, and simply connected.
  \item $\chi(S^2\times S^2)=4=\chi(S^2\mathbin{\wt{\times}} S^2)$.
  \item By the diffeomorphisms $(S^2 \times S^2) \# \CP^2 \cong \CP^2 \# \ol{\CP}^2 \# \CP^2 \cong (S^2 \mathbin{\wt{\times}} S^2)  \# \CP^2$, the manifolds are $\CP^2$-stably diffeomorphic and $\CP^2$-stably homeomorphic.
  \item The second Stiefel-Whitney classes are distinct, since $S^2 \times S^2$ is spin but $S^2 \mathbin{\wt{\times}} S^2$ is not.  So they are neither (stably) diffeomorphic, nor (stably) homeomorphic, nor (simple) homotopy equivalent, $s$- nor $h$-cobordant (in either category).
\end{itemize}

This example shows that it is easy to find manifolds that are $\CP^2$-stably homeomorphic or $\CP^2$-stably diffeomorphic but do not satisfy any of the other equivalence relations.

\subsection{\texorpdfstring{$\CP^2$ and Freedman's $\star\CP^2$}{CP2 and Freedman's *CP2}}\label{subsection:row-3}

Freedman~\cite{F}*{p.~370} constructed the manifold $\star\CP^2$, which he called the \emph{Chern manifold}, as follows. Attach a $2$-handle to $D^4$ along a $+1$-framed trefoil $K$, to obtain the \emph{$1$-trace} of the trefoil. The boundary is a homology sphere $\Sigma$, which bounds a compact, contractible manifold $C$~\citelist{\cite{F}*{Theorem~1.4'}\cite{FQ}*{Corollary~9.3C}}. Cap off the $1$-trace with $C$. The resulting closed $4$-manifold is $\star\CP^2$. The same construction with any knot $K$ with $\Arf(K)=1$ gives rise to a homeomorphic manifold.

\begin{itemize}
  \item  The manifolds $\CP^2$ and $\star\CP^2$ are orientable and simply connected.
  \item $\CP^2$ is smooth. The Kirby-Siebenmann invariant $\ks(\star\CP^2)=\ks(C)=\mu(\Sigma)=\Arf(K)=1$, where $\mu(\Sigma)$ is the Rochlin invariant of~$\Sigma$, by \citelist{\cite{FQ}*{page 165} \cite{Gonzalez-Acuna}}. Therefore, $\star\CP^2$ is not smoothable, even stably.
 \item Since the manifolds have isometric intersection forms, they are homotopy equivalent~\cites{Whitehead-4-complexes,Milnor-simply-connected-4-manifolds} and have equal Euler characteristic $\chi(\CP^2)=3=\chi(\star\CP^2)$. Since the Whitehead group of the trivial group is trivial, the manifolds are also simple homotopy equivalent.
 \item Since the Kirby-Siebenmann invariants are different, they are not $\CP^2$-stably homeomorphic and therefore not homeomorphic, not $s$- or $h$-cobordant, and not stably homeomorphic.
 \item The smooth questions are not applicable to this pair.
\end{itemize}

This example shows that one must restrict to manifolds with the same Kirby-Siebenmann invariant, and moreover ideally smooth manifolds, to find really interesting examples of homotopy equivalent but not homeomorphic manifolds.

\subsection{\texorpdfstring{$\RP^4 \# \CP^2$ and $\mathcal{R} \# \star\CP^2$}{RP4+CP2 and R+*CP2}}\label{subsection:row-4}

Here $\star\CP^2$ is the Chern manifold from \cref{subsection:row-3}.
The manifold~$\mathcal{R}$ was first constructed by Ruberman~\cite{ruberman84}, as follows. By~\citelist{\cite{F}*{Theorem~1.4'}\cite{FQ}*{Corollary~9.3C}}, the Brieskorn sphere $\Sigma(5,9,13)$ bounds a compact, contractible, topological $4$-manifold~$\mathcal{U}$.   As a Seifert fibered manifold, $\Sigma(5,9,13)$ admits an orientation preserving order two self-diffeomorphism $t$ given by the antipodal map on the generic fibres. Since the parameters $5$, $9$, and $13$ are all odd, $t$ is also the antipodal map on the exceptional fibres and $t$ has no fixed points.
The manifold $\mathcal{R}$ is then defined as
\[
\mathcal{R}:=\mathcal{U}/ x\sim t(x) \text{ for }x\in \partial \mathcal{U}.
\]
A similar construction was previously used by Fintushel-Stern~\cite{FS-fake} to construct a manifold homeomorphic, but not diffeomorphic to $\RP^4$, as we describe in \cref{subsection:row-6}. Ruberman's proof that $\mathcal{R}$ is not homeomorphic to $\RP^4$ utilises Rochlin's theorem~\cite{Rochlin} (see also~\citelist{\cite{freedman-kirby}\cite{Kirby_lecturenotes}*{Chapter~XI}}) and the fact that the Rochlin invariant $\mu(\Sigma(5,9,13))=1$. That $\mathcal{R}$ is a homotopy $\RP^4$ follows from the same principle as in~\cite{FS-fake}. Namely, write $S^4$ as $\mathcal{U}\cup_{t} \mathcal{U}$, and observe that $\mathcal{R}$ is the quotient under a free involution. The same construction can be applied to any integer homology sphere $\Sigma$ admitting a free orientation preserving involution and with $\mu(\Sigma)=1$. While it is \textit{a priori} not clear that the outcome is unique up to homeomorphism, this follows from the classification of closed, non-orientable $4$-manifolds with order two fundamental group~\cite{HKT-nonorientable}*{Theorem~3}.

In the literature $\mathcal{R}$ is sometimes denoted $\star\RP^4$. We prefer not to use this notation to avoid confusion with the star construction, defined in \cref{sec:E-and-friends}.  The manifold $\mathcal{R}$ also arises via a surgery construction, which we outline after the following list.

\begin{itemize}
 \item  The manifolds $\RP^4 \# \CP^2$ and $\mathcal{R} \# \star\CP^2$ are nonorientable with nontrivial fundamental group isomorphic to $\Z/2$.
  \item $\RP^4 \# \CP^2$ is smooth by construction. Ruberman-Stern~\cite{Ruberman-Stern} showed that $\mathcal{R} \# \star\CP^2$ is also smoothable, as follows. First they showed that there exists a knot $K$ with $S^3_1(K)=\partial X_1(K)=\Sigma(5,9,13)$ a Brieskorn sphere, via an explicit Kirby calculus argument. Let $t$ be a free orientation preserving involution of $\Sigma(5,9,13)$ as described at the beginning of this section. Construct the smooth $4$-manifold
  \[
  X:= X_1(K)/x\sim t(x) \text{ for }x\in \partial X_1(K).
  \]
  To see that $X$ is homeomorphic to $\mathcal{R} \# \star\CP^2$, let $\mathcal{U}_K$ be the compact, contractible, topological $4$-manifold with boundary $\Sigma(5,9,13)$ provided by~\citelist{\cite{F}*{Theorem~1.4'}\cite{FQ}*{Corollary~9.3C}}. By Freedman \cite{F} and Boyer~\cite{Boyer1}, there is a homeomorphism
  \[\star\CP^2\#\mathcal{U}_K\cong (X_1(K)\cup_{\Sigma(5,9,13)} \mathcal{U}_K)\# \mathcal{U}_K\cong X_1(K),\]
since they are both compact, simply connected $4$-manifolds with the same intersection form and the same integer homology sphere boundary. This homeomorphism descends to a homeomorphism $\star\CP^2\#\mathcal R\cong X$ when quotienting the boundaries by the involution. Alternatively, one can apply the classification theorem of~\cite{HKT-nonorientable} to show that $X$ is homeomorphic to $\mathcal{R}\# \star\CP^2$.
  Different choices of $K$, and thereby the Brieskorn sphere, might give rise to different smooth structures on $\mathcal{R} \# \star\CP^2$, but this is currently an open question.
   \item The manifolds $\RP^4 \# \CP^2$ and $\mathcal{R} \# \star\CP^2$ are homotopy equivalent, and therefore have equal Euler characteristic $\chi(\RP^4 \# \CP^2) = 2 = \chi(\mathcal{R} \# \star\CP^2)$. Since the Whitehead group of $\Z/2$ is trivial, they are also simple homotopy equivalent.
\item Since they are homotopy equivalent and smoothable, they are $\CP^2$-stably homeomorphic and $\CP^2$-stably diffeomorphic.
  \item Hambleton-Kreck-Teichner~\cite{HKT-nonorientable} showed that they are stably homeomorphic, but not homeomorphic.
  \item The manifolds are stably diffeomorphic by~\cite{guide}*{Theorem~12.3}, which states that smooth, nonorientable, compact 4-manifolds with universal cover non-spin, that are stably homeomorphic, are stably diffeomorphic. 
    \item  They are not $s$-cobordant, in either category, since if they were they would be homeomorphic by the topological $s$-cobordism theorem (\cref{thm:scob}). Since the Whitehead group of $\Z/2$ is trivial, they are also not $h$-cobordant in either category.
\end{itemize}

We end this section by giving an alternative construction of the manifold $\mathcal{R}$ following Hambleton-Kreck-Teichner~\cite{HKT-nonorientable}*{pp.~650-1}.
There is a degree one normal map, namely the collapse map $E_8 \to S^4$, with domain the $E_8$ manifold constructed by Freedman~\cite{F}*{Theorem~1.7}.  Connect sum with $\RP^4$ in both domain and codomain to obtain a degree one normal map
\[F \colon \RP^4 \# E_8 \to \RP^4.\]
The surgery kernel is the image of the $E_8$ form under the map $L_4(\Z) \to L_4(\Z[\Z/2],w)$, where $w \colon \Z/2 \to C_2$ is the nontrivial character.  But this is the zero map \cite{Wall-surgery-book}*{Chapter~13A, bottom of page 173}.
By the definition of $L_4(\Z[\Z/2],w)$, this means that, perhaps after stabilising with copies of $S^2 \times S^2$, the surgery kernel is a hyperbolic form. Applying \cite{PRT20}*{Corollary~1.4} to the surgery kernel, there is a homeomorphism
\begin{equation}\label{eq:RP4-stable}
\RP^4 \# E_8 \# \bighash{k} (S^2 \times S^2) \cong X  \# \bighash{4+k} (S^2 \times S^2)
\end{equation}
for some integer $k$ and for some $4$-manifold $X$.  By additivity of the Kirby-Siebenmann invariant (see e.g.~\cite{guide}*{Theorem~8.2})
\[\ks(X) = \ks(X  \# \bighash{4+k} (S^2 \times S^2))
=\ks(\RP^4 \# E_8 \# \bighash{4+k} (S^2\times S^2)) = \ks(E_8) = 1,
\]
whereas $\RP^4$ is smooth and so $\ks(\RP^4)=0$.  We can then define
\[\mathcal{R} := X.\]
By~\cite{HKT-nonorientable}*{Theorem~3} the manifold $X$ is determined up to homeomorphism by the existence of the homeomorphism~\eqref{eq:RP4-stable}, and in particular coincides with Ruberman's construction.
Since we built $\mathcal{R}$ by killing the surgery kernel of a degree one normal map, we have constructed an element of the structure set, $[\mathcal{R} \to \RP^4] \in \mathcal{S}(\RP^4)$, and in particular $\mathcal{R} \simeq \RP^4$.

\subsection{\texorpdfstring{Kreck's examples $K3 \# \RP^4$ and $\bighash{11} (S^2 \times S^2) \# \RP^4$}{Kreck's examples K3+RP4 and 11S2xS2+RP4}}\label{subsection:row-5}

These are a pair of relatively easy to understand exotic 4-manifolds discovered by Kreck~\cite{Kreck-exotic}.  Indeed
this was the first known exotic pair of closed 4-manifolds; the examples of Cappell-Shaneson discussed in \cref{subsection:row-6} were constructed earlier, but they were not shown to be homeomorphic until much later~\cite{HKT-nonorientable}.
The $K3$ surface is a well-known smooth, simply connected 4-manifold. One way to construct it is to first consider $E(1) := \CP^2  \# \bighash{9} \ol{\CP^2}$, which comes with an \emph{elliptic fibration} $F \colon E(1) \to S^2$  (be warned that this is not a Serre or Hurewicz fibration).
Generically the point inverse images are tori with trivial normal bundles $T^2 \times D^2$.  The \emph{fibre sum} of two copies of $E(1)$ is the $K3$ surface:
\[K3 := E(2) := E(1) \#_{T^2} E(1) = (E(1) \sm T^2 \times \mathring{D}^2) \cup_{T^2 \times S^1} (E(1) \sm T^2 \times \mathring{D}^2).\]
We can then construct the manifolds $K3 \# \RP^4$ and $\bighash{11} (S^2 \times S^2) \# \RP^4$.

\begin{itemize}
\item The manifolds $K3 \# \RP^4$ and $\bighash{11}(S^2 \times S^2) \# \RP^4$ are smooth, nonorientable, and have nontrivial fundamental group isomorphic to $\Z/2$. They have equal Euler characteristics $\chi(K3 \# \RP^4) = 23 = \chi(\bighash{11}(S^2 \times S^2) \# \RP^4)$.
  \item As we will explain below, Kreck showed that $K3 \# \RP^4$ and $\bighash{11} (S^2 \times S^2) \# \RP^4$ are homeomorphic but not stably diffeomorphic.
 \item Since they are homeomorphic, they are stably homeomorphic, $\CP^2$-stably homeomorphic, homotopy equivalent, simple homotopy equivalent, and topologically $h$- and $s$-cobordant.
 \item Since they are not stably diffeomorphic, the two 4-manifolds are not diffeomorphic, and neither smoothly $h$-cobordant nor smoothly $s$-cobordant.
 \end{itemize}

\begin{remark}
More generally, Kreck~\cite{Kreck-exotic}*{Theorem~1} showed that there is at least one such example of a pair of homeomorphic but not stably diffeomorphic smooth 4-manifolds for each 1-type $(\pi,w\colon \pi \to C_2)$ with $\pi$ a finitely presented group and $w$ nontrivial.
By Gompf's result (\cref{thm:gompf-stable}), orientable (stably) homeomorphic 4-manifolds are stably diffeomorphic, so this phenomenon only arises for nonorientable manifolds.
\end{remark}

Now we argue why $K3 \# \RP^4$ and $\bighash{11} (S^2 \times S^2) \# \RP^4$ are homeomorphic. The intersection form of $K3$ is isometric to the orthogonal sum of two $E_8$ forms and a rank 6 hyperbolic form.
By the classification of closed, simply connected 4-manifolds up to homeomorphism~\cites{F,FQ}, it follows that $K3$ is homeomorphic to $E_8 \# E_8 \# \bighash{3} (S^2 \times S^2)$, where $E_8$ as before denotes the $E_8$ manifold constructed by Freedman~\cite{F}*{Theorem~1.7}.
Next, observe that there is a unique connected sum of an orientable manifold  with a nonorientable manifold such as $\RP^4$, because any two embeddings of $D^4$ in $\RP^4$ are isotopic.
Using this we have homeomorphisms
\begin{align*}
  \RP^4 \# K3 &\cong \RP^4 \# E_8 \# E_8 \#\bighash{3} (S^2 \times S^2) \\
  &\cong \RP^4 \# \overline{E_8} \# E_8 \#\bighash{3} (S^2 \times S^2)\\
  &\cong \RP^4 \#\bighash{11} (S^2 \times S^2).
\end{align*}
The last homeomorphism uses the classification of closed, simply connected 4-manifolds again. Here the 4-manifolds $\overline{E_8} \# E_8 \#\bighash{3} (S^2 \times S^2)$ and $\bighash{11} (S^2 \times S^2)$ have isometric intersection forms, since they are indefinite and have the same rank, parity, and signature~\cite{MH}*{Theorem~5.3}.

Next we explain the obstruction to stable diffeomorphism. Heuristically, the construction of a homeomorphism above does not work smoothly, because one cannot split the $K3$ surface smoothly, because of Rochlin's theorem that smooth, spin, closed, 4-manifolds have signature divisible by 16~\cite{Rochlin} (see also~\citelist{\cite{freedman-kirby}\cite{Kirby_lecturenotes}*{Chapter~XI}}.
When necessary we use $\RP^{\infty} \simeq B\Z/2 = K(\Z/2,1)$ as a model for the Eilenberg-Maclane space.
Define
\[B:= B\Z/2 \times \BSpin\]
and let $\xi \colon B \to BO$ be the composition
\[B= B\Z/2 \times \BSpin \xrightarrow{\gamma^{\perp} \times p} BO \times BO \xrightarrow{\oplus} BO,\]
where $\gamma^{\perp}$ is the orthogonal complement to the tautological line bundle $\gamma$ over $B\Z/2 \simeq \RP^{\infty}$, $p \colon \BSpin \to BO$ is the standard projection, and $\oplus$ is the map corresponding to the Whitney sum of stable bundles.   Replace $\xi$ by a fibration, and by an abuse of notation denote the resulting homotopy equivalent domain by the same letter $B$ and the new map to $BO$ again by $\xi$.

Let $\Omega_4(B,\xi)$ denote the group of bordism classes of closed, smooth 4-manifolds, equipped with a lift to $B$ of the classifying map $\nu_M \colon M \to BO$ of the stable normal bundle. In other words, $\Omega_4(B,\xi)$ has elements represented by pairs $(M,\widetilde{\nu}_M)$, where $M$ is a closed, smooth $4$-manifold and $\nu_M$ classifies its stable normal bundle, such that the following diagram commutes.
\[\begin{tikzcd}
  & B \ar[d,"\xi"] \\ M \ar[ur,"\wt{\nu}_M"] \ar[r,"\nu_M"] & BO.
\end{tikzcd}\]
A $(B,\xi)$-bordism between $(M,\wt{\nu}_M)$ and $(N, \wt{\nu}_N)$ is a compact, smooth 5-manifold $(W,\wt{\nu}_W)$ with a corresponding lift of the stable normal bundle $\nu_W \colon W \to BO$ to $B$, and a diffeomorphism $i_M \sqcup i_N \colon M \sqcup N \xrightarrow{\cong} \partial W$ such that $\wt{\nu}_W \circ i_M = \wt{\nu}_M$ and $\wt{\nu}_W \circ i_N = \wt{\nu}_N$.

For $M \in \{K3 \# \RP^4, \bighash{11} (S^2 \times S^2) \# \RP^4\}$, Kreck showed that $M$ admits a 2-connected lift $\wt{\nu}_M \colon M \to B$. Roughly speaking, this is because $(B,\xi)$ was chosen to be compatible with the Stiefel-Whitney classes $w_1(\nu_M)$ and $w_2(\nu_M)$, and also with the fundamental group $\pi_1(M)$.
The following theorem is a special case of~\cite{surgeryandduality}*{Theorem~C}.

\begin{theorem}[Kreck~\cite{Kreck-exotic}]\label{thm:kreck-stable-diffeo}
Let $(M,\wt{\nu}_M), (N, \wt{\nu}_N) \in \Omega_4(B,\xi)$ with $\wt{\nu}_M$ and $\wt{\nu}_N$ 2-connected. Then $M$ and $N$ are stably diffeomorphic if and only if
\[[(M,\wt{\nu}_M)] = [(N, \wt{\nu}_N)] \in \Omega_4(B,\xi)/ \hAut(B,\xi),\]
where $\hAut(B,\xi)$ denotes the group of fibre homotopy classes of fibre homotopy equivalences of the fibration $\xi \colon B \to BO$.
\end{theorem}

Kreck showed~\cite{Kreck-exotic}*{Proposition~2} that there is an isomorphism
\begin{equation}\label{eq:alpha-16}
\alpha \colon \Omega_4(B,\xi) \xrightarrow{\cong} \Z/16,
\end{equation}
with $\alpha(\bighash{11}(S^2 \times S^2),s ) = 0$ and $\alpha(K3,s) = 8$.   Here these are simply connected 4-manifolds, and the maps $s$ factor through $\{\pt\} \times \BSpin$, and correspond to the unique spin structures on the respective 4-manifolds.
 The automorphism group  $\hAut(B,\xi)$ acts by isomorphisms of $\Omega_4(B,\xi)\cong \Z/16$ and so preserves these two elements.  Since $\alpha$ is a homomorphism and the connected sum is $(B,\xi)$-bordant to the disjoint union, it follows that $K3 \# \RP^4$ and  $\bighash{11} (S^2 \times S^2)  \# \RP^4$ are distinct in $\Omega_4(B,\xi)/ \hAut(B,\xi)$
and are therefore not stably diffeomorphic. Next we proceed to explain the computation of $\Omega_4(B,\xi)$ and the isomorphism $\alpha$ in a little more detail.

Represent an element of $\Omega_4(B,\xi)$ by a map $f \colon M \to \RP^4 \times \BSpin$, using the cellular approximation theorem and the fact that $B\Z/2\cong \RP^\infty$.
Let $\pr_1 \colon \RP^4 \times \BSpin \to \RP^4$ be the projection, and make $\pr_1 \circ f$ transverse to $\RP^{3} \subseteq \RP^4$. Let $F \subseteq M$ be the inverse image $f^{-1}(\RP^{3}\times\BSpin)$, which is a closed 3-manifold.  The restriction of $f$ to $F$ determines, with a little work, an induced spin structure on $F$ and a map $F \to \RP^{3} \subseteq \RP^{\infty} \cong B \Z/2$. So $F$ determines an element of $\Omega_3^{\Spin}(B\Z/2)$.

Every element $(F,f)$ of $\Omega^{\Spin}_3(B\Z/2)$ bounds some spin $4$-manifold $W$ which admits a branched double covering $\wh W \to  W$ restricting to the double cover $\wh F \to  F$ corresponding to $f$ with branching set a 2-dimensional submanifold $\Sigma$ of $W$.  We will show the existence of such a null-bordism in the next paragraph. Note that $\partial(\nu F)\cong\wh F$, where $\nu F$ is the normal bundle of $F$ in $M$ and that $M_F := M\setminus \nu F$ is spin. Thus we can form the closed oriented spin $4$-manifold $M_W:= M_F \cup_{\wh F} -\wh W$, where the orientation on $M_F$ is induced from the $(B,\xi)$-structure on $M$. Then we define a precursor of the invariant from \eqref{eq:alpha-16}, as
\[\ol{\alpha} (M,f) := \sign(M_W) - \Sigma\cdot \Sigma\in \Z/32,\]
where $\Sigma\cdot \Sigma$ denotes the homological self-intersection number of $\Sigma\subseteq \wh W$.
Kreck showed that $\ol{\alpha}$ is a well-defined invariant of $\Omega_4(B,\xi)$, namely it is unaffected by cobordism over $B$ and is independent of the choice of $W$ and the branching set.
This uses the Atiyah-Singer $G$-signature theorem~\cite{Atiyah-Singer-G-sign}*{Section~6} (see also \cite{Gordon-G-signature}) and Rochlin's theorem~\cite{Rochlin} (see also~\citelist{\cite{freedman-kirby}\cite{Kirby_lecturenotes}*{Chapter~XI}}.
Thus as alluded to above the exotic behaviour can be ultimately be traced back to Rochlin's theorem.

Now we show the existence of a null-bordism $W$, as promised. Kreck showed in~\cite{Kreck-exotic}*{Proposition~3} that $\Omega_3^{\Spin}(B\Z/2) \cong \Z/8$, generated by the pair consisting of $(\RP^3,\operatorname{inc} \times s \colon \RP^3 \to \RP^{\infty} \times \BSpin)$, where $\operatorname{inc}$ is the inclusion and $s$ is some choice of spin structure.  The 3-manifold $\RP^3$ bounds the 4-manifold $V$ obtained by adding a 2-handle to $D^4$ along an unknot with framing coefficient 2. A generator of $\pi_2(V)$ is represented by an embedded 2-sphere. Taking the 2-fold cover of $V$ branched along such an embedding yields a branched covering $\wh{V} \to V$ restricting to the standard nontrivial double cover $S^3 \cong \partial \wh{V} \to \RP^3 \cong \partial V$ on the boundary. By taking boundary connected sums of $V$, we see that $k\RP^3$ bounds a $4$-manifold $W_k=\natural^k V$ with a branched double cover as claimed. For any element $(F,f)$ of $\Omega^{\Spin}_3(B\Z/2)$ we can now construct $W:=W_k\cup_{k\RP^3}W'$, where $W'$ is a spin bordism from $(F,f)$ to $k\RP^3$ over $B\Z/2$. The null-bordism $W$ admits a branched double cover as required in the definition of $\ol{\alpha}$. This completes the description of $\ol{\alpha}$.

Kreck showed~\cite{Kreck-exotic}*{p.\ 256} that $\im \ol{\alpha} = 2\Z/32 \cong \Z/16$, generated by $\ol{\alpha}(\RP^4,\wt{\nu}_{\RP^4})$ for some normal smoothing $\wt{\nu}_{\RP^4} \colon \RP^4 \to B$.
In this case $F=\RP^3$, $W=V$ from above, $\sign(M_F) = \sign(D^4) = 0$ and $\widehat{W} = \widehat{V}$ so $\sign(\wh{W}) =1$ and $\Sigma \cdot \Sigma =1$. Therefore $\sign(M_W) = -1$ and
so $\ol{\alpha}(\RP^4,f) = -2$.  Kreck also showed by analysing the Atiyah-Hirzebruch spectral sequence for $\Omega_4(B,\xi)$ that $|\Omega_4(B,\xi)| \leq 16$. It follows that $\ol{\alpha}$ is an isomorphism onto $2\Z/32 \subseteq \Z/32$.  Finally, $\ol{\alpha}(K3,s) = 16$ since $F= \emptyset$ and so $M_F = K3$. Since $\im \ol{\alpha} = 2\Z/32$ we may define $\alpha := \ol{\alpha}/2 \colon  \Omega_4(B,\xi) \xrightarrow{} \Z/16$, which gives the isomorphism $\alpha$ that we have been trying to explain.

\subsection{\texorpdfstring{$\RP^4$ and the Cappell-Shaneson exotic $\RP^4$}{RP4 and the Cappell-Shaneson exotic RP4}}\label{subsection:row-6}

Cappell and Shaneson~\cites{CS-new-four-mflds-bams,CS-new-four-mflds-annals} constructed a smooth manifold $R$ which they showed is homotopy equivalent, but not stably diffeomorphic, to $\RP^4$; later work \cite{HKT-nonorientable} showed that $R$ is homeomorphic to $\RP^4$. Fintushel and Stern~\cite{Fintushel-Stern-exotic-free} constructed a smooth manifold $R_{FS}$ with the same properties as a quotient of $S^4$ by an exotic free action of $\Z/2$ on $S^4$. See~\cite{Aitchison-Rubinstein}*{Theorem~5.1} for the relationship between the two constructions.

The Fintushel-Stern construction is easier to describe, so we start with that.  Start with two copies of a compact, contractible, smooth $4$-manifold $U$ with boundary the Brieskorn homology sphere $\partial U  \cong \Sigma(3,5,19)$.   This homology sphere is a Seifert-fibred $3$-manifold.  The antipodal map on the $S^1$ fibres induces a fixed-point free, order-two self-diffeomorphism $t \colon \partial U \to \partial U$, as in~\cref{subsection:row-4}.
Fintushel-Stern showed that there is a diffeomorphism $S^4 \cong U \cup_{t} U$. Switching the two $U$ factors of the union gives rise to a smooth free involution $\wt{t} \colon S^4 \to S^4$. Define
\[
R_{FS}:= S^4/\wt{t} \equiv U/ x\sim t(x) \text{ for }x\in \partial U.
\]
This inspired the construction of $\mathcal{R}$ in~\cite{ruberman84} given in \cref{subsection:row-4}.

Next we recall the construction of the Cappell-Shaneson $R$. Let
\[A := \begin{pmatrix}
                                                         0  & 1 & 0 \\
                                                         0 & 0 & 1 \\
                                                         -1 & 1 & 0
                                                       \end{pmatrix}.\]
The matrix $A$ induces a diffeomorphism $\phi_A \colon T^3 \to T^3$. Consider the punctured 3-torus $T^3_0$ and the corresponding mapping torus for $A$, denoted~$M_{A,0}$.
Take $\RP^4$ and remove a neighbourhood $S^1 \mathbin{\wt{\times}} D^3$ of an embedded circle representing the generator of $\pi_1(\RP^4)\cong\Z/2$.  Then
\[R := (\RP^4 \sm S^1 \mathbin{\wt{\times}} D^3) \cup_{\partial} M_{A,0}.\]

\begin{itemize}
  \item The manifolds $R$ and $\RP^4$ are smooth and nonorientable. They have nontrivial fundamental group isomorphic to $\Z/2$.
  \item They have equal Euler characteristics $\chi(R)=1=\chi(\RP^4)$.
  \item Cappell-Shaneson~\cite{CS-new-four-mflds-annals} showed that $R$ and $\RP^4$ are homotopy equivalent, and therefore simple homotopy equivalent.
  \item After Freedman's work it was later shown~\cite{HKT-nonorientable} that these manifolds are furthermore homeomorphic.
  \item
By combining with their earlier work on $4$-dimensional surgery~\cite{CS-on-4-d-surgery}*{Theorem~2.4}, Cappell-Shaneson~\cite{CS-new-four-mflds-annals}*{p.~61} showed that, in addition to not being diffeomorphic, $R$ and $\RP^4$ are not stably diffeomorphic.
Alternatively, this can be seen using the $\alpha$ invariant, defined below, which turns out to be a stable diffeomorphism invariant, via the reduction of stable diffeomorphism to bordism over the normal 1-type. We refer to \cref{subsection:row-12}, page~\pageref{alpha-stable-diff-invariant},  where we explain this in more detail in the case of Akbulut's examples.
  \item Since they are homeomorphic they are stably homeomorphic, simple homotopy equivalent, topologically $h$- and $s$-cobordant, $\CP^2$-stably homeomorphic and $\CP^2$-stably diffeomorphic.
 \item Since they are not stably diffeomorphic they are not diffeomorphic. They are not smoothly $h$- or $s$-cobordant because they are not stably diffeomorphic.
 \item As mentioned above, the Fintushel-Stern 4-manifold $R_{FS}$ has exactly the same properties.
\end{itemize}

We now describe the diffeomorphism obstruction from \cite{CS-new-four-mflds-annals} used to show that $R$ is not diffeomorphic to $\RP^4$.
Consider the nonorientable linear $S^3$-bundle over $S^1$, which we denote by $S^1\mathbin{\wt{\times}} S^3$.
Cappell-Shaneson defined an invariant
\[
\alpha \colon \mathcal{N}(S^1\mathbin{\wt{\times}} S^3) \to \Z/32\]
 from the set of \emph{smooth} degree one normal maps with target $S^1\mathbin{\wt{\times}} S^3$.
Let $f \colon R \to S^1\mathbin{\wt{\times}} S^3$ represent an element of  $\mathcal{N}(S^1\mathbin{\wt{\times}} S^3)$.
Define $\Sigma_R$ to be a framed $3$-manifold $f^{-1}(S^3)$ obtained by  making $f$ transverse to a fibre $S^3$ in $S^1\mathbin{\wt{\times}} S^3$, taking the inverse image, and pulling back the framing.
Let $\mu(\Sigma_R) \in \Z/16$ be the Rochlin invariant of $\Sigma_R$, by definition the signature mod 16 of a framed 4-manifold with boundary $\Sigma_R$.
In addition, write $W:=R\sm \nu \Sigma_R$, and let $\sigma(W) \in\Z$ be its signature.
Then Cappell-Shaneson defined
  \[
  \alpha(R,f):=2\mu(\Sigma_R)-\sigma(W)\mod{32}.
  \]
In \cite{CS-new-four-mflds-annals}*{Proposition~2.1} they showed that $\alpha$ is a well-defined map $\alpha \colon \mathcal{N}(S^1\mathbin{\wt{\times}} S^3) \to \Z/32$ as claimed.  This invariant will be used again in \cref{subsection:row-12}; see also Kreck's invariant in \cref{subsection:row-5}, which was inspired by the Cappell-Shaneson invariant.

Let $M_A$ be the mapping torus for the diffeomorphism $\phi_A\colon T^3\to T^3$ used in the construction of $R$. In \cite{CS-new-four-mflds-annals}*{Propositions~2.1~and~2.2} it was shown that there is a degree one normal map $f \colon M_A \to S^1\mathbin{\wt{\times}} S^3$ and that $\alpha(M_A,f)$ is nonzero.  As such, just as in \cref{subsection:row-5}, the exotic behaviour is due to Rochlin's theorem. In \cite{CS-new-four-mflds-annals}*{Theorem~3.1}, Cappell and Shaneson use the nontriviality of $\alpha(M_A,f)$ to show that $R$ is not diffeomorphic to $\RP^4$.  This theorem shows that there is a homotopy equivalence $h \colon R \to \RP^{4}$ that is not homotopic to a diffeomorphism. Then since every homotopy equivalence of $\RP^4$ is homotopic to the identity (cf.\ \cref{prop:stab-diff}), it follows that there is no diffeomorphism between $R$ and $\RP^4$. It was shown in~\cite{gompf-killing} that the double cover of $R$ is diffeomorphic to $S^4$.

The Cappell-Shaneson construction can be varied by making different choices for the matrix $A$, giving rise to an exotic $\RP^4$ denoted $R_A$. The precise conditions are that $A \in \GL(3,\Z)$ with $\det(A) =-1$ and $\det(I-A^2) = \pm 1$; call such a matrix a \emph{Cappell-Shaneson matrix}. It can be seen from the construction that if $A$ and $A'$ are similar Cappell-Shaneson matrices then $R_A$ and $R_{A'}$ are diffeomorphic. The universal cover of $R_A$ is a smooth homotopy 4-sphere. Many of these have been shown to be diffeomorphic to $S^4$ -- we refer to the introduction of~\cite{Kim-Yamada} for a detailed survey. The most general result in this vein is that the universal cover of $R_A$ arising from a Cappell-Shaneson matrix $A$ with trace $n$ where $-64\leq n\leq 69$ is diffeomorphic to $S^4$ [ibid.]. However, this corresponds to only finitely many similarity classes,
and so there are infinitely many Cappell-Shaneson homotopy $4$-spheres which remain as potential counter-examples to the smooth 4-dimensional Poincar\'{e} conjecture.

\subsection{\texorpdfstring{$L \times S^1$ and $L' \times S^1$}{LxS1 and L'xS1}}\label{subsection:row-7}

Let $L$ and $L'$ be 3-dimensional lens spaces with the same fundamental group that are homotopy equivalent but not homeomorphic.  These are of the form $L_{p,q_1}$ and $L_{p,q_2}$ with $\gcd(p,q_i) =1$, and  $1 \leq q_i <p$ for $i=1,2$, such that for some $n$,  $q_1q_2 \equiv \pm n^2 \mod{p}$ (for homotopy equivalent), and $q_i\not\equiv\pm q_j^{\pm 1}\mod{p}$ (for non-homeomorphic)~~\cite{reidemeister}. See also~\cites{brody,davis-kirk,Cohen}.

\begin{itemize}
  \item The manifolds $L \times S^1$ and $L' \times S^1$ are smooth and orientable. They have fundamental group isomorphic to $\Z/p\times \Z$, for some $p$.
  \item They are homotopy equivalent because $L$ and $L'$ are, and since they are smooth they are therefore $\CP^2$-stably homeomorphic and $\CP^2$-stably diffeomorphic.
    \item The formula for the Whitehead torsion of a product of homotopy equivalences~\citelist{\cite{kwun-szczarba}*{Corollary~1.3}\cite{lueck-surgery-intro}} implies that $L \times S^1$ and $L' \times S^1$ are simple homotopy equivalent. Indeed, let $f \colon L \to L'$ be a homotopy equivalence. Let $i \colon \Z \to \Z/p \times \Z$ and let $j \colon \Z/p \to \Z/p \times \Z$ be the standard inclusions.  Then
  \begin{equation}\label{it:lens-4}
    \tau(f \times \Id) = j_*(\tau(f)) \cdot \chi(S^1) + \chi(L) \cdot i_*(\tau(\Id)) =  0,
  \end{equation}
  since $\chi(S^1)=0 = \chi(L)$, where $i_*$ and $j_*$ are the induced maps on Whitehead groups.
 \item
 In the late 1960s it was proven that $L \times S^1$ and $L' \times S^1$ are not diffeomorphic.
     If they were, then $L$ and $L'$ would be smoothly $h$-cobordant, which can be seen by embedding $L$ in the infinite cyclic cover $L' \times \RR$.  Atiyah-Bott~\cite{Atiyah-Bott}*{Theorem~7.27} and Milnor~\cite{Milnor-whitehead-torsion}*{Corollary~12.12} showed that smoothly $h$-cobordant lens spaces are homeomorphic.  Therefore~$L \times S^1$ and $L' \times S^1$ are not diffeomorphic.
 \item In the late 1980s, Turaev~\cite{Turaev1988} showed that moreover $L \times S^1$ and $L' \times S^1$ are not homeomorphic.  He showed that for any 3-manifolds $M$ and $M'$ that do not fibre over $S^1$ with periodic monodromy, the product $M \times S^1$ and $M' \times S^1$ are homeomorphic if and only if $M$ and $M'$ are homeomorphic~\cite{Turaev1988}*{Theorem~1.5}.  To show this he proved that such $M$ and $M'$ are topologically $h$-cobordant if and only if  $M$ and $M'$ are homeomorphic. It follows that $L \times S^1$ and $L' \times S^1$ are not homeomorphic. Of course this also reproves that they are not diffeomorphic.
  \item The manifolds $L \times S^1$ and $L' \times S^1$ are stably diffeomorphic, and therefore stably homeomorphic.
  We prove this in \cref{prop:stab-diff} below, by adapting a proof of Cappell-Shaneson~\cite{Cappell-Shaneson-nonlinear-similarity}. That this is possible was stated in \cite{Weinberger-higher-rho}.
\item By the topological $s$-cobordism theorem (\cref{thm:scob}), $L \times S^1$ and $L' \times S^1$ are not topologically $s$-cobordant and are therefore not smoothly $s$-cobordant. Here we use that $\Z/p\times \Z$ is a good group.
 \item They are not topologically $h$-cobordant, as we explain in \cref{prop:lens-space-x-S1-not-h-cobordant} using~\cite{Kwasik-Schultz}. It follows that they are not smoothly $h$-cobordant.
 \item By choosing a large enough value of $p$, one may use the classification of lens spaces to find arbitrarily large, finite sets $\{L_i\times S^1\}_i$, such that the elements pairwise satisfy the above properties.
   \end{itemize}

\begin{proposition}\label{prop:stab-diff}
	Let $L$ and $L'$ be 3-dimensional lens spaces that are homotopy equivalent but not homeomorphic. Then the $4$-manifolds $L \times S^1$ and $L' \times S^1$ are stably diffeomorphic.
\end{proposition}

\begin{proof}
The strategy is as follows. Let $\pi_1(L)\cong \pi_1(L')\cong \Z/p$ and let $h\colon L\to L'$ be a homotopy equivalence. Let $\mathcal{S}^h(L')$ be the homotopy structure set of $L'$ and consider the map in the surgery sequence $\eta \colon \mathcal{S}^h(L') \to \mathcal{N}(L') \cong [L',\G/\TOP]$ with target the normal invariants of $L'$. This map is defined for 3-manifolds, even though there is no analogue of the entire surgery sequence for 3-manifolds (but see \cite{Kirby-Taylor}*{Theorem~4} for a version with a homology structure set).
We will show that the homotopy equivalences $h \colon L \to L'$ and $\Id \colon L' \to L'$ determine equal elements $\eta(h) = \eta(\Id) \in \mathcal{N}(L')$ i.e.\ normally bordant degree one normal maps.   Crossing with $S^1$ we see that $\eta(h \times \Id) = \eta(\Id \times \Id) \in \mathcal{N}(L' \times S^1)$. Since $\Z/p \times \Z$ is a good group, the surgery sequence is exact. Therefore $[h \times \Id \colon L \times S^1 \to L' \times S^1]$ and $[\Id \times \Id \colon L' \times S^1 \to L' \times S^1]$  are in the same orbit of the action of $L_5^h(\Z[\Z/p \times \Z])$ on $\mathcal{S}^h(L' \times S^1)$. Then, by the definition of the Wall realisation $L_5^h$ action, it follows that $L \times S^1$ and $L' \times S^1$ are stably homeomorphic. Then since these 4-manifolds are smooth and orientable, they are in fact stably diffeomorphic by~\cref{thm:gompf-stable}.

  We therefore have to show that $\eta(h) = \eta(\Id) \in \mathcal{N}(L')$. For this we adapt the proof of \cite{Cappell-Shaneson-nonlinear-similarity}*{Proposition~2.1}, where the corresponding fact for lens spaces with even order fundamental group was proven in all dimensions, under an additional hypothesis that the double covers are homeomorphic. We will show that in dimension 3 the extra hypotheses are not needed.

  First recall that there is a 4-connected map $k \colon \G/\TOP \to K(\Z/2,2)$, corresponding to a universal cohomology class $k \in H^2(\G/\TOP;\Z/2)$~\citelist{\cite{Madsen-Milgram}\cite{Kirby-Taylor}}.
This induces a homomorphism
\[k_*\colon [L',\G/\TOP]\to [L',K(\Z/2,2)]\cong H^2(L';\Z/2).\]
Consider $\eta(h)$ and $\eta(\Id)$ as elements of $[L',\G/\TOP]$. Since $\eta(\Id)=0$, also $k_*(\eta(\Id)) =0 \in H^2(L';\Z/2)$. So we have to show that $k_*(\eta(h)) = 0$, and then we will have shown that both maps $L'\to \G/\TOP$ are null-homotopic, and hence that the two lens spaces are normally bordant.
    If $p$ is odd, then $H^2(L';\Z/2) \cong H_1(L';\Z/2) =0$, so we are done. We therefore assume that $p$ is even, in which case $H^2(L';\Z/2) = \Z/2$, and we have something to check. So let $p = 2r$, for some $r \geq 1$.

  Now we diverge from the proof of~\cite{Cappell-Shaneson-nonlinear-similarity}.  Let $f \colon \wt{L} \to L$ and $f' \colon \wt{L'} \to L'$ be the $r$-fold covers so that $\pi_1(\widetilde{L})\cong \Z/2\cong\pi_1(\widetilde{L'})$. Note that $\wt{L}$ and $\wt{L'}$ are again lens spaces, so $\wt{L} \cong \wt{L'} \cong \RP^3 = L_{2,1}$.
  We claim that $(f')^* \colon H^2(L';\Z/2) \to H^2(\RP^3;\Z/2)$ is an isomorphism.
  Note that $f'_*\colon H_1(\RP^3;\Z)\to H_1(L';\Z)$ is given by multiplication with $r$. Hence on the cellular $\Z$-chain complexes, $f'_1$ is given by multiplication with $r$ and thus $f'_2$ is the identity as can be seen by considering the following commutative diagram of the cellular chain complexes over $\Z$.
  \[\begin{tikzcd}
  	\Z\ar[r,"0"]\ar[d,"f'_3"']&\Z\ar[r,"\cdot 2"]\ar[d,"\cdot 1","f'_2"']&\Z\ar[r,"0"]\ar[d,"\cdot r","f'_1"']&\Z\ar[d,"\cdot 1","f'_0"']\\
  	\Z\ar[r,"0"]&\Z\ar[r,"\cdot 2r"]&\Z\ar[r,"0"]&\Z.
  	\end{tikzcd}\]
  The claim that $(f')^* \colon H^2(L';\Z/2) \to H^2(\RP^3;\Z/2)$ is an isomorphism immediately follows from this.
By the commutative square
\[\begin{tikzcd}
{[L',\G/\TOP]}\ar[r,"(f')^*"]\ar[d,"k_*"]&{[\RP^3,\G/\TOP]}\ar[d,"k_*"]\\
H^2(L';\Z/2)\ar[r,"(f')^*","\cong"']&H^2(\RP^3;\Z/2)
\end{tikzcd}\]
in order to prove $k_*(\eta(h))=0$, it suffices to show that $(f')^*(\eta(h))=0$. Since $(f')^*\colon \cN(L')\to \cN(\RP^3)$ is given by pulling back along $f'$, we have $(f')^*(\eta(h))=\eta(\wt h)$, where
   $\wt{h} \colon \RP^3 \to \RP^3$ is obtained from lifting $h \circ f \colon \RP^3 \to L'$ to $\RP^3$ along $f'$ as in the diagram
   \[\begin{tikzcd}
     \RP^3 \ar[r,dashed,"\wt{h}"] \ar[d,"f"] & \RP^3 \ar[d,"f'"] \\ L \ar[r,"h"] & L'.
   \end{tikzcd}\]
  We assert that every orientation-preserving homotopy self-equivalence of $\RP^3$, and so in particular $\wt{h}$, is homotopic to the identity. It then follows that $\wt{h}$ is trivial in the structure set of $\RP^3$, and so $\eta(\wt{h}) = \eta(\Id_{\RP^3})=0$ as desired.

It remains to prove the assertion that every orientation-preserving homotopy self-equivalence of $\RP^3$ is homotopic to the identity. This can be proven via obstruction theory, by iteratively extending a map defined on $\RP^3\times \{0,1\}$ to $\RP^3\times [0,1]$. Since the target $\RP^3$ is path connnected, there is no obstruction to extending over the relative $1$-cells of $\RP^3\times [0,1]$, i.e.\ to defining the homotopy on the $0$-cells of $\RP^3$.  For $k\geq 2$, the obstruction to extending over the relative $k$-cells of $\RP^3 \times [0,1]$ lies in \[H^{k}(\RP^3 \times [0,1],\RP^3 \times\{0,1\};\pi_{k-1}(\RP^3)).\]
  For $k=2$, the obstruction vanishes because both maps induce the identity on $\pi_1(\RP^3)\cong\Z/2$. Since $\pi_2(\RP^3) = \pi_2(S^3) =0$, the remaining obstruction lies in
  \[H^{4}(\RP^3 \times [0,1],\RP^3 \times\{0,1\};\pi_{3}(\RP^3)) \cong  H_0(\RP^3 \times [0,1];\Z) \cong \Z.\]
  The obstruction measures the difference in the degrees of the two maps. Since both are degree 1, the obstruction vanishes and the assertion is proved.
\end{proof}

	\begin{proposition}\label{prop:lens-space-x-S1-not-h-cobordant}
		Let $L$ and $L'$ be homotopy equivalent lens spaces that are not homeomorphic.
		The manifolds $L\times S^1$ and $L'\times S^1$ are not topologically $h$-cobordant.
	\end{proposition}
	
	\begin{proof}
  		Assume that there is a topological $h$-cobordism $W$ from $L\times S^1$ to $L'\times S^1$.
		As in \eqref{it:lens-4} above, taking a product with $S^1$ kills the Whitehead torsion, i.e.\ $W\times S^1$ is an $s$-cobordism from $L\times S^1\times S^1$ to $L'\times S^1\times S^1$. The high-dimensional $s$-cobordism theorem then implies that $L \times S^1 \times S^1$ and $L' \times S^1 \times S^1$ are homeomorphic. But this implies that $L$ and $L'$ are homeomorphic by the toral stability property for lens spaces \cite{Kwasik-Schultz} which can be seen using higher $\rho$-invariants \cite{Weinberger-higher-rho}. This is a contradiction to our assumption on $L$ and $L'$.
	\end{proof}

Indeed if $L$ and $L'$ are not homeomorphic, then even $L\times \R^2$ and $L'\times\R^2$ are not homeomorphic \cite{Kwasik-Schultz}*{Theorem~1.4}.

\subsection{\texorpdfstring{Donaldson's examples $E(1)$ and the Dolgachev surface $E(1)_{2,3}$}{Donaldson's examples E(1) and the Dolgachev surface  E(1)(2,3)}}\label{subsection:row-8}

As mentioned before, Kreck (\cref{subsection:row-5}) and Cappell-Shaneson (\cref{subsection:row-6}) constructed the first examples of exotic $4$-manifolds. These were nonorientable, and the obstructions used arose from Rochlin's theorem.

New examples of exotic pairs, including simply connected examples, were provided by Donaldson~\cite{Donaldson:1987-1}, and many others after him (see e.g.\ \cites{GompfStip,akbulut-book}). Donaldson's first examples consisted of $E(1) = \CP^2 \# \bighash{9} \ol{\CP}^2$ and the \emph{Dolgachev surface} $E(1)_{2,3}$, which is obtained from $E(1)$ via two log transforms. Let us recall the construction.
The 4-manifold $E(1)$ admits the structure of an elliptic fibration $f \colon E(1) \to S^2$. Let $T^2 \subseteq E(1)$ be a generic fibre. Its normal bundle is a copy of $T^2 \times D^2$ embedded in $E(1)$.

In general, a \emph{log transform} is a surgery operation on a 4-manifold $X$ with a smoothly  embedded torus $T^2 \subseteq X$ with trivial normal bundle which cuts out a neighbourhood $T^2 \times D^2$ of $T^2$, and glues it back via a diffeomorphism of $\partial(T^2 \times D^2) = T^3$.   Let $\{\alpha, \beta, [\partial D^2]\}$ be a basis for $H_1(T^3)$. Then by definition we reglue to form
\[X' := \ol{X \sm (T^2 \times D^2)} \cup_{\varphi} (T^2 \times D^2)\]
using a diffeomorphism $\varphi_p \colon T^3 \to T^3$, for some $p \in \Z$, corresponding to an element
\[\begin{pmatrix}
  1 & 0 & 0 \\ 0 & 0 & 1 \\ 0 & -1 & p
\end{pmatrix}\]
of $\operatorname{GL}(3,\Z)$.
Perform two of these log transform operations on $E(1)$, on disjoint generic fibres of the elliptic fibration, one with $p=2$ and one with $p=3$.  The resulting 4-manifold is the Dolgachev surface $E(1)_{2,3}$. One can also construct $E(1)_{2,3}$ by a single \emph{knot surgery} operation~\citelist{\cite{FS-ICM}\cite{FS-knot-surgery}\cite{GompfStip}*{Section~10.3}\cite{akbulut-book}*{Section~6.5}} on a generic fibre of $E(1)$, using a trefoil knot~\citelist{\cite{park-noncomplex}*{p.\ 7}\cite{FS-lectures}*{Lecture~6, Section~2}}. Akbulut used this in~\cite{Akbuklut-kebab} to obtain a description of the Dolgachev surface without $1$- or $3$-handles.

\begin{itemize}
\item $E(1)$ and $E(1)_{2,3}$ are smooth, closed, orientable, and simply connected.
\item They have isometric intersection forms $\langle +1\rangle \oplus 9\langle -1\rangle$ and so are homeomorphic~\cites{F,FQ}, and therefore (simple) homotopy equivalent, stably homeomorphic, topologically $h$- and $s$-cobordant, and $\CP^2$-stably homeomorphic.
\item They are homeomorphic and orientable and therefore stably diffeomorphic and $\CP^2$-stably diffeomorphic (\cref{thm:gompf-stable}).
\item They are smoothly $h$-cobordant by~\cite{Wall-on-simply-conn-4mflds}*{Theorem~2}, and are therefore smoothly $s$-cobordant since the Whitehead group of the trivial group is trivial.
\item They are not diffeomorphic~\cite{Donaldson:1987-1}, via tools of Yang-Mills gauge theory.
\end{itemize}

While Rochlin's theorem suffices to construct exotic pairs of nonorientable $4$-manifolds (see \cref{subsection:row-5,subsection:row-6}) it seems that one requires the full force of gauge theory to detect orientable exotic pairs. After Donaldson's work, Seiberg-Witten theory provided an easier, but nonetheless still highly nontrivial, way to distinguish manifolds such as $E(1)$ and $E(1)_{2,3}$ which are related by log transforms or knot surgery.  There is a large literature on generalisations of Donaldson's example, as described in~ \cites{GompfStip,akbulut-book}. As stated before, since this is not our focus we restrict ourselves to recalling the first known example.


\subsection{\texorpdfstring{$\bighash{3} E_8$}{3E8} and the Leech manifold}\label{subsection:row-9}

Next we present 4-manifolds that are stably homeomorphic but not homotopy equivalent.
Freedman~\cite{F} showed that every nonsingular, symmetric, integral bilinear form can be realised as the intersection pairing of a closed, simply connected, topological $4$-manifold.
The forms $\oplus^3 E_8$ and the Leech lattice are even, symmetric, positive definite bilinear forms of rank and signature 24, so they are realised by closed $4$-manifolds that we denote by $\bighash{3} E_8$ and $Le$ respectively.

\begin{itemize}
  \item  The manifolds $\bighash{3} E_8$ and $Le$ are simply connected and orientable with $\chi(\bighash{3} E_8)= 26 = \chi(Le)$.
  \item Since they are spin with signature $24$, Rochlin's theorem implies that the manifolds both have nontrivial Kirby-Siebenmann invariant, and are therefore not smoothable.
 \item The manifolds have inequivalent intersection pairings and are therefore not homotopy equivalent. As a result they are neither simple homotopy equivalent, homeomorphic, topologically $h$-cobordant, nor topologically  $s$-cobordant.
 \item Since they are spin and the Euler characteristics and the signatures coincide, $\bighash{3} E_8$ and $Le$ are stably homeomorphic and therefore also $\CP^2$-stably homeomorphic, as follows. The stable classification of closed, simply connected, spin topological 4-manifolds is essentially due to Wall~\cite{Wall-on-simply-conn-4mflds}*{Theorems~2~and~3}: two such 4-manifolds are stably homeomorphic if and only if there are choices of orientations with respect to which the manifolds are equivalent in the topological spin bordism group $\Omega_4^{\TOPSpin} \cong \Z$, with the isomorphism given by $[M] \mapsto \sigma(M)/8$.  Wall worked in the smooth category, but the analogous topological category result is straightforward to deduce~\cite{KPT}*{Section~2.2}.
 \item The smooth questions are not applicable to this pair.
\end{itemize}

The downside of this example is that the manifolds are not smoothable.   It turns out that this is inevitable when considering simply connected 4-manifolds, as shown by the next proposition.

\begin{proposition}\label{prop:simply-conn-hom-equiv-iff-stably-diff}
Closed, smooth, simply connected 4-manifolds $M$ and $N$ with equal Euler characteristics are stably diffeomorphic if and only if they are homotopy equivalent.
\end{proposition}

\begin{proof}
Let $M$ and $N$ be closed, smooth, and simply connected $4$-manifolds. Assume that $M$ and $N$ are homotopy equivalent. Then they have isometric intersection forms, for some choice of orientation, so by
 \cite{Wall-on-simply-conn-4mflds}*{Theorem~2} they are smoothly $h$-cobordant, and by \cite{Wall-on-simply-conn-4mflds}*{Theorem~3} (see also~\cref{prop:stab-diff}) they are stably diffeomorphic.

For the other direction, assume $M$ and $N$ are stably diffeomorphic and $\chi(M)=\chi(N)$. Then modulo changing orientations, $\sigma(M)=\sigma(N)$. Since $\chi(M)=\chi(N)$ and $\sigma(M)=\sigma(N)$, the intersection forms of $M$ and $N$ are either both definite or both indefinite. In the definite case, the intersection forms must be diagonal by Donaldson's theorem~\cite{Donaldson}, and so the intersection forms are isometric,  and therefore the manifolds are homotopy equivalent~\cites{Whitehead-4-complexes,Milnor-simply-connected-4-manifolds}.
 For the indefinite case, note that since the hyperbolic form is even, and the intersection forms of $M$ and $N$ become isometric after stabilising, they are either both odd or both even.
Indefinite forms are determined up to isometry by the rank, parity, and signature~\cite{MH}*{Theorem~5.3} and so again~$M$ and~$N$ are homotopy equivalent.
\end{proof}

Finally, we note that any pair of even, inequivalent, nonsingular, symmetric, integral, bilinear forms with equal rank and signature could have been used in this section to produce a pair of stably homeomorphic but not homotopy equivalent manifolds.

\subsection{Kreck--Schafer manifolds}\label{subsection:row-10}

Kreck and Schafer~\cite{Kreck-Schafer} constructed smooth 4-manifolds that are stably diffeomorphic but not homotopy equivalent. As observed in \cref{prop:simply-conn-hom-equiv-iff-stably-diff}, their examples are necessarily not simply connected.

They  used the  following general construction.  Start with a finite presentation of a group. Form the corresponding presentation 2-complex $X$. Thicken it to a 5-dimensional manifold $N(X)$, e.g.\ by embedding $X$ in $\R^5$ and letting $N(X)$ denote a smooth regular neighbourhood~\cite{thickenings}. Then consider the 4-manifold $\partial N(X)$. One can use this to find a 4-manifold with any given finitely presented fundamental group.

For any two finite 2-complexes $X$ and $X'$ with the same fundamental group, there are integers $m,n$ such that $X \vee^m S^2 \simeq X' \vee^n S^2$~\cite{lms197}*{(40)}.  It follows that the boundaries of the 5-dimensional thickenings $\partial N(X)$ and $\partial N(X')$ are stably diffeomorphic.

Kreck and Schafer used finite 2-complexes $X$ and $X'$ with the same fundamental group, as above, that have the same Euler characteristic but are not homotopy equivalent. Finding examples of 2-complexes with this property is rather difficult, but examples are known~\cites{Metzler,Sieradski,Lustig}.
Kreck and Schafer's obstruction applies for many nontrivial fundamental groups, the smallest of which is $\Z/5 \times \Z/5\times \Z/5$.
Kreck and Schafer then showed that for their particular choices of $X$ and $X'$, the 4-manifolds $\partial N(X)$ and $\partial N(X')$ are not homotopy equivalent.  These manifolds have the additional interesting property that their intersection forms, and indeed their equivariant intersection forms, are hyperbolic.

\begin{itemize}
\item The manifolds $\partial N(X)$ and $\partial N(X')$ are smooth, non-simply-connected, oriented manifolds with the same Euler characteristic that are stably diffeomorphic but not homotopy equivalent.
\item Since they are not homotopy equivalent, they are also not simple homotopy equivalent, nor homeomorphic, nor diffeomorphic, nor $h$- or $s$-cobordant in either category.
\item Since they are stably diffeomorphic they are stably homeomorphic, $\CP^2$-stably diffeomorphic, and $\CP^2$-stably homeomorphic.
\end{itemize}

Kreck and Schafer found pairs of $4$-manifolds with the properties listed.
Are there stable diffeomorphism classes of smooth, oriented $4$-manifolds containing infinitely many homotopy equivalence classes, all with the same Euler characteristic?  Or even arbitrarily many?

\subsection{Teichner's \texorpdfstring{$E \# E$ and $*E \# \star E$}{E+E and *E+*E}}\label{subsection:row-11}
\label{sec:E-and-friends}

A \emph{star partner} of a $4$-manifold $M$ is a manifold $\star M$ such that there exists a homeomorphism $M\#\star\CP^2\cong \star M\#\CP^2$ preserving the decomposition on $\pi_2$, where $\star \CP^2$ is the Chern manifold whose construction we recalled in \cref{subsection:row-3}. Let $E$ denote the unique fibre bundle over $\RP^2$ with fibre $S^2$, that has orientable but not spin total space. We give a Kirby diagram in~\cref{fig:E}. This is a smooth, closed, orientable 4-manifold with fundamental group $\Z/2$. Teichner~\cite{Teichner:1997-1} showed that $E$ has a star partner $\star E$ which is simple homotopy equivalent to $E$ but has opposite Kirby-Siebenmann invariant. This will also follow from the more general \cref{starhomeo-are-homotopy-equivalent,lemma:star-partner-exists} below. By the surgery exact sequence, if $\pi_1(M)\cong\Z/2$, then $\star M$ is unique up to homeomorphism if it exists (see also \cite{Teichner:1997-1}*{Theorem~1}). In particular, this means that $\star E$ is the unique star partner for $E$.

\begin{figure}[htb]
	\centering
\begin{tikzpicture}
        \node[anchor=south west,inner sep=0] at (0,0){	\includegraphics{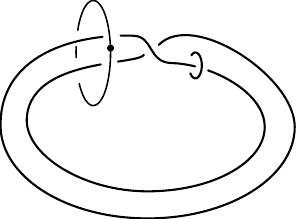}};
		\node at (3.25,2.1) {$0$};
		\node at (4,3.1) {$1$};
	\end{tikzpicture}
\caption{A Kirby diagram for the manifold $E$, the unique fibre bundle over $\RP^2$ with fibre $S^2$, that has orientable but not spin total space.}
\label{fig:E}
\end{figure}

\begin{itemize}
  \item The manifolds $E \# E$ and $\star E \# \star E$ are orientable, with nontrivial fundamental group isomorphic to $\Z/2 * \Z/2$.
  \item The manifold $E$ is smooth by construction, and therefore so is $E \# E$.  The manifold $\star E$ is not smoothable, but it is currently open whether $\star E\# \star E$ is smoothable.  Indeed $\star E\# \star E$ has vanishing Kirby-Siebenmann invariant and so is stably smoothable, i.e.\ there exists a $k$ such that $\star E\# \star E  \# \bighash{k} (S^2 \times S^2)$ is a smooth manifold.
  \item Teichner~\cite{Teichner:1997-1}*{Proposition~3} showed that  $E \# E$ and $\star E \# \star E$ are not stably homeomorphic, and therefore they are not homeomorphic and not $s$- or $h$- cobordant.
 \item They are simple homotopy equivalent, since $E$ and $\star E$ are simple homotopy equivalent. This will also follow from \cref{starhomeo-are-homotopy-equivalent} later, since we will show that they are both star partners of $E\# \star E$. Recall also that the Whitehead group of the infinite dihedral group is trivial~\cite{stallings-freeproducts}, and so it also suffices to know that they are homotopy equivalent.
 \item Since they are homotopy equivalent and have trivial Kirby-Siebenmann invariant they are $\CP^2$-stably homeomorphic. The homeomorphism from \eqref{eq:EE-CP2stable} below gives an alternative argument, and shows that only one $\CP^2$-factor is required.
 \item For smooth examples with the same properties, consider the pair $E\# E \# \bighash{k} (S^2 \times S^2)$ and $\star E\# \star E  \# \bighash{k} (S^2 \times S^2)$. As mentioned above, there exists $k$ for which these are smooth since $\ks(E\#E)=\ks(\star{E}\#\star{E})=0$. These two manifolds are still simple homotopy equivalent, $\CP^2$-stably homeomorphic and therefore $\CP^2$-stably diffeomorphic, but not stably homeomorphic, and therefore not stably diffeomorphic, neither smoothly $s$- nor $h$-cobordant, and not diffeomorphic.
\end{itemize}

Taking connected sums of $E\# E$ and $\star E \# \star E$, as well as sufficiently many copies of $S^2 \times S^2$, one can construct homotopy equivalence classes containing arbitrarily many stable homeomorphism classes of smooth, orientable $4$-manifolds using the techniques of~\cite{Teichner:1997-1}. We omit the details.  In \cite{teichner-thesis}*{Example~5.2.4}, Teichner also constructed similar examples for finite fundamental groups with quaternionic 2-Sylow subgroup.

As a counterpoint to these examples we show that it is impossible to find infinite families of $4$-manifolds that are all homotopy homotopy equivalent but pairwise not stably homeomorphic. For this proof, we will need the following terminology.
The \emph{normal $1$-type} of a smooth $4$-manifold $M$ is a fibration $\xi\colon B \to  BO$, inducing an injection $\pi_2(B) \to \pi_2(BO)$ and an isomorphism on $\pi_i(B) \to \pi_i(BO)$ for $i>2$, which further admits a lift $\wt\nu_M \colon M \to B$ of the stable normal bundle $\nu_M\colon M\to BO$, inducing an isomorphism $\pi_1(B) \to \pi_1(BO)$ and a surjection $\pi_2(B) \to \pi_2(BO)$. A choice of a lift $\wt\nu_M$ is called a \emph{normal $1$-smoothing} of $M$. For a normal $1$-type $(B,\xi)$, let $\Omega_4(B,\xi)$ denote the group of bordism classes of normal $1$-smoothings.
For topological $4$-manifolds, we have parallel notions of a topological normal $1$-type $B\to \BTOP$ and topological normal $1$-smoothings lifting the stable topological normal bundle. For a topological normal $1$-type $(B,\xi)$, let $\Omega_4^{\TOP}(B,\xi)$ denote the group of topological bordism classes of topological normal $1$-smoothings.

\begin{proposition}\label{prop:finite-stab-homeo}
The set of stable homeomorphism types of closed $4$-manifolds in a fixed homotopy type is finite.
Moreover the set of stable diffeomorphism types of closed, smooth $4$-manifolds in a fixed homotopy type is finite.
\end{proposition}

\begin{proof}
Let $M$ be a closed $4$-manifold with $\pi:=\pi_1(M)$ and orientation character $w$. We will use that the composition $\mathcal{S}^h(M)\xrightarrow{\eta} \mathcal{N}(M)\xrightarrow{\sigma} L_4^h(\Z\pi,w)$ in the surgery sequence is trivial, in both the smooth and topological categories, and with no restriction on fundamental groups.

First we give the proof in the topological category.
We claim that the map $\eta\colon \mathcal{S}^h(M)\to \mathcal{N}(M)$ has finite image.
Recall from \cref{Section:review-surgery} that $\mathcal{N}(M)\cong H^4(M;\Z)\oplus H^2(M;\Z/2)$. Here we see that $H^2(M;\Z/2)$ is finite and by Poincar\'e duality $H^4(M;\Z)\cong H_0(M;\Z^w)$. So when $M$ is orientable, we have that $H_0(M;\Z^w)\cong H_0(M;\Z)\cong \Z$, which maps injectively into $L_4^s(\Z\pi,w)$ under the surgery obstruction map $\sigma$.
Consider the degree one normal map $f\colon M\# \bighash{k} E_8\to M$ given by the collapse map, where as before $E_8$ denotes the manifold constructed by Freedman~\cite{F}*{Theorem~1.7}. Then under the augmentation map $L_4^h(\Z\pi) \to L_4^h(\Z)$, $\sigma([M\# \bighash{k} E_8,f])$ maps to $k \in \Z \cong L_4(\Z)$.  It follows that the image of $\sigma$ is infinite and therefore since $\sigma$ is a homomorphism, the kernel of $\sigma$ is finite.
When $M$ is nonorientable, $H_0(M;\Z^w)\cong \Z/2$, and so $\mathcal{N}(M)$ is already finite.
This completes the proof of the claim.

To complete the proof in the topological category we show that
two elements $(N,f), (N',f')\in \mathcal{S}^h(M)$ with equal image in $\mathcal{N}(M)$ are stably homeomorphic. Let $(B,\xi)$ denote the topological normal $1$-type of $M$ and let $\wt\nu_M$ be a topological normal $1$-smoothing. Then $\wt\nu_M\circ f$ and $\wt\nu_M\circ f'$ are normal $1$-smoothings for $N$ and $N'$ respectively, and moreover $(N,\wt\nu_M\circ f)$ and $(N',\wt\nu_M\circ f')$ are equal in $\Omega_4^{\TOP}(B,\xi)$ by hypothesis. By \cite{surgeryandduality}*{Theorem~C}, the manifolds $N$ and $N'$ are stably homeomorphic. This completes the proof of the first statement.

Now assume that $M$ is a smooth, closed 4-manifold.  If $M$ is orientable, then by \cite{Gompf-stable} (see also \cite{guide}*{Theorems~12.13}) every pair of stably homeomorphic smooth 4-manifolds is stably diffeomorphic, and so we are done. Suppose that $M$ is nonorientable.
As in the topological case, it suffices to show that the set $\cN^\Diff(M)$ of smooth normal invariants is finite.
Since $\PL/\OO$ is 6-connected, $\cN^\Diff(M)\cong[M,\G/ \OO]\cong [M,\G/\PL]\cong \mathcal{N}^{\PL}(M)$, and so it suffices to show that $\mathcal{N}^{\PL}(M)$ is finite.
The fibre sequence
$\TOP/\PL \to \G/\PL \to \G/\TOP$ induces an exact sequence of sets \[[M,\TOP/\PL] \to [M,\G/\PL] \to [M,\G/\TOP],\]
which translates to
\[H^3(M;\Z/2) \to \mathcal{N}^{\PL}(M) \to \mathcal{N}(M).\]
The first and last terms are finite sets, and therefore so is $\mathcal{N}^{\PL}(M) \cong \mathcal{N}^{\Diff}(M)$, as desired.
\end{proof}

Next we take the opportunity to prove some basic facts about the star construction, some of which were used in the discussion at the start of the section. We show that star partners are simple homotopy equivalent and that the star partnership relation is symmetric.
Then we discuss uniqueness of star partners, and we give a criterion that guarantees star partners exist.

\begin{lemma}\label{starhomeo-are-homotopy-equivalent}
Let $M$ be a $4$-manifold with a star partner $\star M$. Then $M$ and $\star M$ are simple homotopy equivalent.
\end{lemma}

\begin{proof}
We use an argument due to Stong~\cite{Stong-conn-sum}*{Section 2}. For a 4-manifold $M$ and $\beta \in \pi_2(M)$, let $\capp(M,\beta)$ be the result $M \cup_{\beta} D^3$ of adding a 3-cell to $M$ along $\beta$.  Let $h \colon \star M \# \CP^2 \to M \# \star \CP^2$ be a homeomorphism preserving the decomposition of $\pi_2$. Let $\alpha \in \pi_2(\CP^2)\cong \Z$ be a generator. Then $h_*(\alpha)$ generates $\pi_2(\star \CP^2)$.  Further $\capp(\CP^2,\alpha) \simeq_s S^4$ and $\capp(*\CP^2,h_*(\alpha)) \simeq_s S^4$. Therefore
  \begin{align*}
    M &\cong M \# S^4 \simeq_s M \# \capp(\star\CP^2,h_*(\alpha)) \cong \capp(M\#\star\CP^2,h_*(\alpha)) \cong \capp(\star M\# \CP^2,h_*^{-1}\circ h_*(\alpha)) \\ &\cong \capp(\star M\# \CP^2,\alpha) \cong \star M\# \capp(\CP^2,\alpha) \simeq_s \star M \# S^4 \cong \star M
    \end{align*}
as desired.
\end{proof}

Note that $\star M$  and $M$ have opposite Kirby--Siebenmann invariants, by additivity of the Kirby-Siebenmann invariant under connected sum~\cite{guide}*{Theorem~8.2}.

In the next proof we will need the following fact. If the universal cover $\wt M$ is non spin and the fundamental group of $M$ is good, then $\star M$ is unique up to homeomorphism \cite{Stong-conn-sum}*{Corollary~1.2}.

\begin{proposition}\label{prop:star-symmetric}
	Let $M$ be a $4$-manifold with a star partner $\star M$. Suppose that $\pi_1(M)$ is good. Then the relation of being a star partner is symmetric, i.e.\ $M$ is a star partner of $\star M$.
\end{proposition}

\begin{proof}
We must show that there is a homeomorphism $M\# \CP^2 \cong \star M\#\star \CP^2$, preserving the decomposition on $\pi_2$. 

First we show that $M\# \CP^2$ is a star partner for $M\# \star \CP^2$.
By the classification of closed, simply connected 4-manifolds~\cites{F,FQ}, there is a homeomorphism \[\star\CP^2\#\star\CP^2\cong \CP^2\#\CP^2\] preserving the decomposition on $\pi_2$. Therefore, $(M\# \CP^2)\#\CP^2\cong (M\# \star \CP^2)\# \star \CP^2$, preserving the decomposition on $\pi_2$. This shows that $M\# \CP^2$ is a star partner for $M\# \star \CP^2$.

Since $\star M$ is a star partner of $M$, we see that $\star M\# \CP^2\cong M\# \star \CP^2$, preserving the decomposition on $\pi_2$. By taking a connected sum on both sides with $\star\CP^2$, we see that $\star M\#\star \CP^2$ is a star partner of $M\#\star\CP^2$.

Thus both $M\#\CP^2$ and $\star M\#\star\CP^2$ are star partners for $M\#\star \CP^2$.
Next we apply the uniqueness of star partners for manifolds with non-spin universal covers and good fundamental group mentioned above.
More precisely, we apply \cite{Stong-conn-sum}*{Theorem~1.1\,(b)}, using $\CP^2$ as the closed 1-connected 4-manifold, $W_1 = M\#\CP^2$, and $W_2 = \star M\#\star\CP^2$. By the previous two paragraphs we have homeomorphisms
\[\CP^2 \# (M\#\CP^2) \cong \star \CP^2 \# (M \# \star \CP^2) \cong \CP^2 \# (\star M\#\star\CP^2), \]
and Stong's theorem gives us the desired homeomorphism $M\#\CP^2 \xrightarrow{\cong}\star M\#\star\CP^2$ that preserves the decomposition on $\pi_2$.
\end{proof}

\begin{example}
 Teichner showed in~\cite{Teichner:1997-1} that uniqueness of star partners does not hold for non-spin manifolds with spin universal covers. In particular, there are homeomorphisms
\begin{equation}\label{eq:EE-CP2stable}
E\# (E \# \CP^2) \cong E\# (\star E \# \star \CP^2)= \star E\# (E \# \star \CP^2)\cong \star E\#(\star E\# \CP^2)
\end{equation}
where we have used that $E\# \CP^2 \cong \star E\# \star \CP^2$ by~\cref{prop:star-symmetric}, and that $E\# \star \CP^2\cong \star E\# \CP^2$, by the definition of $\star E$. This shows that both $E\# E$ and $\star E\# \star E$ are star partners of $E\# \star E$.
\end{example}

Next we give a general criterion for when star partners exist. For example, this can be used to establish the existence of $\star E$. The proof will use the following notion.

\begin{definition}
Let $M$ be a $4$-manifold. An immersion of a $2$-sphere $\alpha \colon S^2 \looparrowright M$ is said to be \emph{$\RP^2$-characteristic} if for every immersion $R \colon \RP^2 \looparrowright M$ such that $R^* w_1(M) = 0$, we have $\alpha \cdot R \equiv R\cdot R \in \Z/2$.
\end{definition}

The following consequence of work of Stong~\cite{Stong} is probably well-known to the experts, but has not appeared in print before.

\begin{proposition}\label{lemma:star-partner-exists}
Let $M$ be a 4-manifold with good fundamental group and containing an immersion $R \colon \RP^2 \looparrowright M$ such that $R\cdot R\equiv 1 \mod 2$ and $R^* w_1(M) = 0$. Then a star partner $\star M$ exists.
\end{proposition}

\begin{proof} The manifold $\star M$ can be constructed as follows. Start with $M\#\star\CP^2$ and let $\alpha$ be an immersed sphere in $\star\CP^2$ with trivial self-intersection number, $\mu(\alpha)=0$, representing a generator of $\pi_2(\star \CP^2)$. Note that $\alpha$ is self-dual. In the construction of $\star\CP^2$ from \cref{subsection:row-3}, one can find such an $\alpha$ by gluing together the track of a null-homotopy for the trefoil in $D^4$ with the core of the attached $2$-handle, and then adjusting the self-intersection number by adding small cusps.
Since the mod 2 intersection numbers are such that  $\alpha\cdot R=0\neq R\cdot R \in \Z/2$, $\alpha$ is not $\RP^2$-characteristic. Stong \cite{Stong}*{p.~1310} proved that in this setting, where $\pi_1(M)$ is good, and $\alpha$ admits an algebraically dual immersed sphere but is not $\RP^2$-characteristic, then $\alpha$ is homotopic to an embedding $\alpha'$.
Since $\lambda(\alpha',\alpha')=1$, it follows that $\alpha'$ has a regular neighbourhood with boundary $S^3$. Consequently, $M\#\star \CP^2\cong N\#\CP^2$, where $N$ is obtained from $M\#\star \CP^2$ by replacing a regular neighbourhood of $\alpha'$ by $D^4$. By construction, $N$ is a star partner for $M$, which we denote by $\star M$.
\end{proof}

\subsection{\texorpdfstring{Akbulut's exotic $(S^1 \mathbin{\wt{\times}} S^3) \# (S^2 \times S^2)$}{Akbulut's exotic S1 ~xS3 + S2xS2}}\label{subsection:row-12}

In~\citelist{\cite{Akbulut-on-fake}*{\S3}\cite{akbulut-fake-gluck}} (see also~\cite{akbulut-book}*{Section~9.5}), Akbulut constructed a smooth, closed $4$-manifold $P$ that is homotopy equivalent to $Q:= (S^1 \mathbin{\wt{\times}} S^3) \# (S^2 \times S^2)$, but not diffeomorphic to $Q$. We give the construction presently. A manifold with similar properties was first constructed by Akbulut in~\cite{Akbulut-fake-4-manifold}. We say more about that and other alternative constructions at the end of the section.

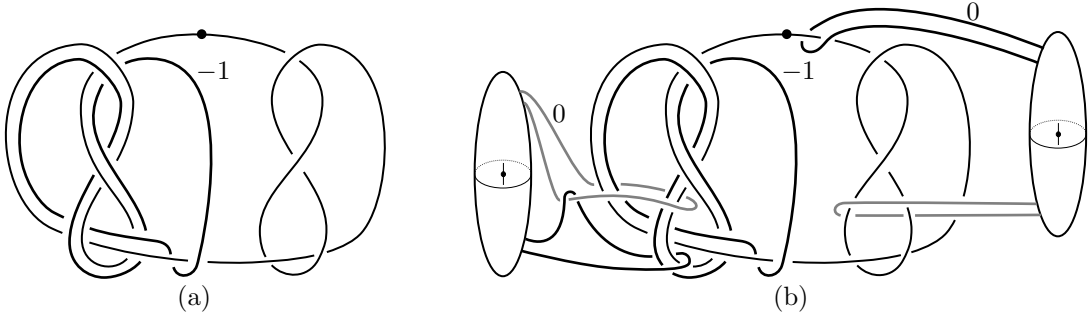
\begin{figure}[htb]
	\centering
\begin{tikzpicture}
        \node[anchor=south west,inner sep=0] at (0,0){	\includegraphics{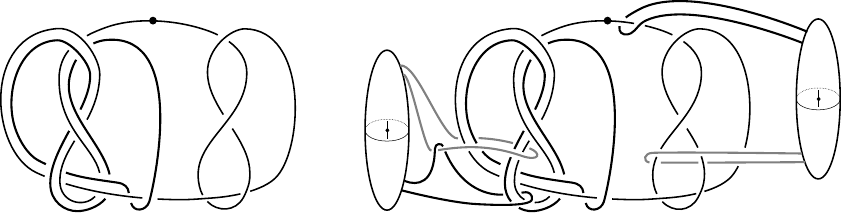}};
		\node at (2.5,-0.25) {(a)};
		\node at (10.35,-0.25) {(b)};
		\node at (2.75, 2.75) {$-1$};
		\node at (10.45, 2.75) {$-1$};
		\node at (12.75, 3.55) {$0$};
		\node at (7.3,2.2) {$0$};
	\end{tikzpicture}
\caption{(a) A Kirby diagram for the product $\Sigma\times [0,1]$, where $\Sigma$ denotes the $(2,3,7)$-Brieskorn sphere. The handle decomposition has a $3$-handle that is not pictured. The dotted circle indicates the complement in $D^3\times [0,1]$ of the product concordance from the figure eight knot to itself. A $-1$-framed $2$-handle is then attached along this concordance. The resulting handlebody has boundary~$\Sigma \# -\Sigma$. The $3$-handle is attached along the connected sum sphere to finish the construction. Here we have used that $\Sigma$ is the result of $-1$-framed Dehn surgery on the figure eight knot. (b) A nonorientable $1$-handle is attached, along with two $2$-handles. The attaching sphere of the new $1$-handle is shown using the notation of \citelist{\cite{Akbulut-fake-4-manifold}\cite{akbulut-book}*{Section~1.5}}: they are identified via the orientation preserving diffeomorphism $(x,y,z)\mapsto (x,-y,-z)$, with respect to coordinates based at the centre. One of the new $2$-handles is shown in grey to help the reader distinguish between the two $2$-handles. It is required, but can be checked, that the $2$-handles are attached along curves disjoint from the attaching $2$-sphere of the $3$-handle from (a). The manifold $P$ is formed by gluing on $D^4$.}
\label{fig:akbulut-brieskorn}
\end{figure}
Let $\Sigma$ denote the $(2,3,7)$-Brieskorn sphere. Recall that the Rochlin invariant of $\Sigma$, denoted by $\mu(\Sigma)$, is by definition the signature mod 16 of a smooth, spin 4-manifold with boundary $\Sigma$. In this case $\mu(\Sigma)\equiv 8\mod{16}$.  Attach a nonorientable $1$-handle to $\Sigma \times [0,1]$, joining the two boundary components, followed by a pair of $2$-handles, as shown in~\cref{fig:akbulut-brieskorn}. One then checks that the boundary of this new handlebody is $S^3$,
so it can be capped off with $D^4$, yielding the desired manifold~$P$.

\begin{itemize}
  \item The manifolds $P$ and $Q$ are smooth, closed, and nonorientable. They have nontrivial fundamental group $\Z$.
  \item The Euler characteristics are equal: $\chi(P)=2=\chi(Q)$.
  \item The equivariant intersection form of $P$ can be computed from its handle description. Wang's classification~\cite{Wang-nonorientable} of nonorientable $4$-manifolds with fundamental group $\Z$ up to homeomorphism then shows that $P$ and $Q$ are homeomorphic, since they have isomorphic equivariant intersection forms and equal Kirby-Siebenmann invariant. Here $P$ and $Q$ both have trivial Kirby-Siebenmann invariant since they are smooth.
  \item Since they are homeomorphic, $P$ and $Q$ are topologically $h$-cobordant, and consequently $s$-cobordant since the Whitehead group of $\Z$ is trivial. 
    \item \label{alpha-stable-diff-invariant} They are not stably diffeomorphic, which implies that they are not smoothly $h$-cobordant; the latter fact was shown by Akbulut~\cite{Akbulut-on-fake}*{Theorem~3}. His proof can be adapted to show the stronger fact that $P$ and $Q$ are not stably diffeomorphic, as we explain next.

        Let $c \colon Q \# \bighash{k} (S^2 \times S^2) \to S^1\mathbin{\wt{\times}} S^3$ be the map collapsing all $S^2 \times S^2$ factors.
        Akbulut constructed a homotopy equivalence~$f \colon P\# \bighash{k} (S^2 \times S^2) \to Q \# \bighash{k} (S^2 \times S^2)$, for all $k$, such that the Brieskorn sphere
         $\Sigma$ is the inverse image under $g:=c \circ f\colon P \# \bighash{k} (S^2 \times S^2) \to S^1\mathbin{\wt{\times}} S^3$ of a copy of $S^3$ in $S^1\mathbin{\wt{\times}} S^3$ and such that \[\sigma((P\# \bighash{k} (S^2 \times S^2)) \sm \nu \Sigma) = \sigma(S^3 \times [0,1] \# \bighash{k+1} (S^2 \times S^2)) =0.\]
        We can consider $(P \# \bighash{k} (S^2 \times S^2),g)$ as a degree one normal map in $\mathcal{N}(S^1\wt\times S^3)$, as usual modulo smooth normal bordism.

        We compute the Cappell-Shaneson $\alpha$-invariant $\alpha \colon \mathcal{N}(S^1\mathbin{\wt{\times}} S^3) \to \Z/32$, described in \cref{subsection:row-6}:
      \[\alpha(P\# \bighash{k} (S^2 \times S^2) ,g)=2\mu(\Sigma)-\sigma((P\# \bighash{k} (S^2 \times S^2)) \sm \nu \Sigma) \equiv 2\cdot 8 - 0 \equiv 16\mod{32}.\]
If $P$ and $Q$ were stably diffeomorphic, then we can use a diffeomorphism $h \colon Q \# \bighash{k} (S^2 \times S^2) \to P \# \bighash{k} (S^2 \times S^2)$ to obtain an degree one normal map $(Q\# \bighash{k} (S^2 \times S^2),g \circ h)$. Since $h$ is a diffeomorphism we would have $\alpha(Q\# \bighash{k} (S^2 \times S^2),g \circ h) = \alpha(P\# \bighash{k} (S^2 \times S^2),g) = 16$. However Akbulut also computed that $\alpha(Q \# \bighash{k} (S^2 \times S^2),\ell)=0$ for every degree one normal map $\ell \colon Q \# \bighash{k} (S^2 \times S^2) \to S^1\mathbin{\wt{\times}} S^3$. It follows that $P$ and $Q$ are not stably diffeomorphic.

\item Since $P$ and $Q$ are not smoothly $h$-cobordant, they not smoothly $s$-cobordant, nor diffeomorphic.
  \item   They are $\CP^2$-stably diffeomorphic and $\CP^2$-stably homeomorphic, because they are homotopy equivalent and both have vanishing Kirby-Siebenmann invariant. Moreover, via explicit handle manipulation~\citelist{\cite{Akbulut-on-fake}*{Theorem~1}\cite{akbulut-fake-gluck}\cite{akbulut-book}*{Exercise~9.3}}, one sees that $P$ is the result of a Gluck twist on an embedded $2$-sphere in $Q$. This shows that not only are $P$ and $Q$ $\CP^2$-stably diffeomorphic, but in fact $P\#\CP^2 \cong Q\# \CP^2$~\cite{Akbulut-on-fake}*{Corollary~2}.
\end{itemize}

Notably the pair $P$ and $Q$ comprise the first example where the Gluck twist operation on a 2-sphere changes the smooth structure of a $4$-manifold. Whether this is possible in the orientable setting remains open. Other examples of the operation changing the smooth structure on nonorientable $4$-manifolds are given in~\citelist{\cite{torres-gluck}\cite{KPR-gluck}*{Proposition~1.6}}.  See also~\cite{akbulut-yasui} for a condition that implies the Gluck twist operation does not change the diffeomorphism type.

As mentioned above, another manifold with similar properties as $P$ was constructed by Akbulut in~\cite{Akbulut-fake-4-manifold}. That manifold, which we call $P'$, also has an explicit handle decomposition~\cite{Akbulut-fake-4-manifold}*{Figure~4.6} consisting of one $0$-handle, one $1$-handle, two $2$-handles, one  $3$-handle, and one $4$-handle. Since $P'$ is nonorientable, the $1$-handle is necessarily nonorientable. In other words, $P'$ is obtained by attaching two $2$-handles to $S^1\mathbin{\wt{\times}} D^3$ and then capping off with another copy of $S^1\mathbin{\wt{\times}} D^3$. Akbulut showed using explicit moves on the handle decompositions that
  \[
  P'_0\cup_\partial (\RP^2\wt\times D^2)\cong R \# (S^2\times S^2),\]
where $P'_0:= P'\sm \Int{S^1\wt\times D^3}$ and $R$ is the Cappell-Shaneson exotic $\RP^4$ from \cref{subsection:row-6}. The double cover of $P'$ is the standard $(S^1 \times S^3) \# \bighash{2} S^2\times S^2$, as shown in~\cite{torres-gluck}*{Proposition~9}, using the fact that the double cover of $R$ is diffeomorphic to $S^4$~\cite{gompf-killing}. Akbulut's construction from~\cite{Akbulut-fake-4-manifold} can be modified to use other Cappell-Shaneson $\RP^4$s, some of which are not known to have standard double covers. Akbulut showed that $P'$ is homotopy equivalent to $Q$. As with~$P$, the classification result of Wang~\cite{Wang-nonorientable} shows that~$P'$ is homeomorphic to~$Q$. Akbulut used explicit moves on the handle decompositions to show that~$P'$ is not diffeomorphic to $Q$, reducing the problem to the fact that the Cappell-Shaneson exotic $\RP^4$ from \cref{subsection:row-6} is not diffeomorphic to $\RP^4$. In has been claimed~\citelist{\cite{Akbulut-on-fake}*{Theorem~1}\cite{akbulut-fake-gluck}} that $P$ and $P'$ are diffeomorphic, but a proof has so far not appeared.

Another construction of a manifold homeomorphic but not diffeomorphic to $Q$ was given by Fintushel-Stern in~\cite{FS-fake}, using the technology of~\cite{Fintushel-Stern-exotic-free}. By surgering an exceptional fibre of the $(3,5,19)$-Brieskorn sphere, they constructed $K$, a homology $S^2\times S^1$. They then formed $X$, the mapping cylinder of the quotient map $K\to K/t$ where $t$ is the free involution contained in the $S^1$-action on $K$, as in~\cref{subsection:row-4,subsection:row-6}. To finish the construction Fintushel and Stern showed that $K$ is the boundary of a homotopy $(S^1\times D^3)\#(S^2\times S^2)$, whose union with $X$ is the desired manifold~$M$. Using the invariant defined in~\cite{Fintushel-Stern-exotic-free}, they showed that $M$ is not smoothly $s$-cobordant to $Q$. They also showed using the handle decomposition, and the fact that $t$ is isotopic to the identity, that the double cover of $M$ is diffeomorphic to the standard $(S^1 \times S^3) \# \bighash{2} S^2\times S^2$. It is not known whether Akbulut's $P$ is diffeomorphic to $M$.

\subsection{Kwasik-Schultz manifolds homotopy equivalent to \texorpdfstring{$L \times S^1$}{LxS1}}\label{subsection:row-13}

The existence of these manifolds is the content of \cref{intro-thm-1} from the introduction, which was first proven in~\cite{Kwasik-Schultz}*{Theorem~1.2}. We restate the theorem and give an original proof below.

\begin{theorem}\label{intro-thm-1-repeated}
Let $M := L \times S^1$, where $L$ is a lens space $L_{p,q}$ with $p \geq 2$, $1\leq q < p$, and $(p,q)=1$. Then there is an infinite collection of closed, orientable, topological $4$-manifolds $\{M_i\}_{i=1}^\infty$, that are all simple homotopy equivalent to $M$ but pairwise not homeomorphic.
\end{theorem}

\begin{proof}
The proof will use the simple surgery exact sequence.
The simple $L$-group satisfies
\[L_5^s(\Z[\Z/p \times \Z]) \cong \Z^r \oplus (\text{torsion}),\] where \[r = \begin{cases}
  (p+1)/2 & p \text{ odd} \\
  (p+2)/2 & p \text{ even}.
\end{cases}\]
In both cases $r>1$ since $p \geq 2$.
  To compute these $L$-groups, first use Shaneson splitting~\cite{Shaneson-splitting} to obtain \[L_5^s(\Z[\Z/p \times \Z]) \cong L_5^s(\Z[\Z/p]) \oplus L_4^h(\Z[\Z/p]).\] Then $L_4^h(\Z[\Z/p]) \cong \Z^r \oplus T$, where $T$ is a torsion group and the free part is detected by a multi-signature invariant: see \cites{Bak-even-L-groups-of-odd-order-groups,Bak-BAMS} for $p$ odd,  \cite{Bak-computation-L-groups}*{Theorem~2} for $p=2^k$, and \cite{Hambleton-Taylor-guide}*{p.~227~and~Proposition~12.1} for the deduction of the general case. On the other hand $L_5^s(\Z[\Z/p])=0$, as shown in~\cites{Bak-odd-L-groups-of-odd-order-groups,Bak-BAMS} and ~\cite{Hambleton-Taylor-guide}*{Theorem~10.1} for $p$ odd, \cite{Bak-computation-L-groups}*{Theorem~7} for $p=2^k$, and again \cite{Hambleton-Taylor-guide}*{p.~227~and~Proposition~12.1} for general~$p$.

Since $\Z/p \times \Z$ is a good group, the simple surgery sequence is exact.
The normal maps $\mathcal{N}(M \times [0,1], M \times \{0,1\})$ are given by the direct sum of $H^2(M \times [0,1],M\times \{0,1\};\Z/2)$ and
\[H^4(M \times [0,1],M\times \{0,1\};\Z)\cong H_1(M\times[0,1];\Z)\cong \Z\oplus \Z/p,\]
as in~\eqref{eq:cohomology-normalinvariants}.
In particular, the normal maps have rank $1$. Hence the quotient
\[L_5^s(\Z[\Z/p \times \Z])/ \sigma(\mathcal{N}(M \times [0,1],M\times \{0,1\}))\]
is infinite. By exactness this quotient acts freely on the structure set $\mathcal{S}^s(M)$, and so the structure set of~$M$ is also infinite.
In order to complete the proof, we need to consider the manifold set:
\[\mathcal{M}(M) := \{N \text{ a closed 4-manifold} \mid N \simeq_s M\}/\text{homeomorphism.}\]
This set is isomorphic to the simple structure set $\mathcal{S}^s(M)$ modulo the action of the simple homotopy self-equivalences of $M$. We will show that the group $\hAut(M)$ of homotopy classes of homotopy self-equivalences of $M$ is finite in \cref{lemma:hom-self-equiv-finite} below.  It follows that the group of simple homotopy self-equivalences $\hAut^s(M)$ is also finite.   Then $\mathcal{M}(M)$ is the quotient of an infinite set $\mathcal{S}^s(M)$ by a finite group, so is again infinite. The elements of $\mathcal{M}(M)$ comprise the manifolds $\{M_i\}_{i=1}^\infty$ in the theorem statement.
\end{proof}

\begin{itemize}
\item The elements of $\mathcal{M}(M)$ are orientable and have nontrivial fundamental group isomorphic to $\Z/p\times \Z$.
\item They are in general not known to be smoothable, but their Kirby-Siebenmann invariants vanish because they are all bordant to the smooth manifold $M$.  For $L'$ homotopy equivalent, but not homeomorphic, to $L$, the smooth 4-manifold $L' \times S^1$ lies in $\mathcal{M}(M)$. In particular, we know that $L'\times S^1$ is simple homotopy equivalent to $M$ by \eqref{it:lens-4}. However, such examples account for at most finitely many of the elements of $\mathcal{M}(M)$.   Therefore the smooth equivalence relations are not applicable in general.
\item As they lie in the same simple structure set, they are all homotopy equivalent and simple homotopy equivalent to one another, and therefore in particular all have equal (vanishing) Euler characteristics.
\item Since they are obtained from the action of $L_5^s(\Z[\Z/p \times \Z])$, the elements of $\mathcal{M}(M)$ are stably homeomorphic and $\CP^2$-stably homeomorphic.
\item The elements of $\mathcal{M}(M)$ are by definition pairwise non-homeomorphic. The cardinality of $\mathcal{M}(M)$ was first shown to be infinite in~\cite{Kwasik-Schultz}*{Theorem~1.2}. As a result, since $\Z/p\times \Z$ is a good group, they are also not topologically $s$-cobordant, by the $s$-cobordism theorem (\cref{thm:scob}).
\item An infinite subset of the manifolds in $\mathcal{M}(M)$ are in addition not topologically $h$-cobordant to one another. To see this we argue as follows.
 Infinitely many of the elements of $L_5^s$ that we used, namely those in the $\Z^r$ summand detected by multisignatures, are nontrivial under the forgetful map $L_5^s(\Z[\Z/p \times \Z]) \to L_5^h(\Z[\Z/p \times \Z]) \cong \oplus^r \Z$.  
This can be seen directly from the definition of multisignatures or by observing that the Rothenberg exact sequence \cite{Shaneson-splitting}*{Proposition~4.1} implies that the kernel of the map $L^s_n(R)\to L^h_n(R)$ is 2-torsion for every ring with involution $R$ and every $n\in\Z$.
It follows that the quotient of these elements by the image of the $\Z$ factor in the normal invariants,
\[L_5^h(\Z[\Z/p \times \Z]) / \sigma(\mathcal{N}(M\times [0,1],M\times \{0,1\})) \cong \oplus^{r-1} \Z,\]
 also act nontrivially on the homotopy structure set $\mathcal{S}^h(M)$. As above $r \geq 2$ so this is infinite. Recall that the equivalence relation defining this set is topological $h$-cobordism over ~$M$.  The quotient of $\mathcal{S}^h(M)$ by the group of homotopy self-equivalences of $M$ is the manifold $h$-cobordism set:
 \[\mathcal{M}^h(M) := \{N \text{ a closed 4-manifold} \mid N \simeq M\}/h\text{-cobordism}.\]
 As before, the homotopy self-equivalences form a finite group, so can only identify finitely many of the manifolds. It follows that there is an infinite subset of $\mathcal{M}(M)$ represented by manifolds that determine distinct elements of $\mathcal{M}^h(M)$, and are therefore pairwise not topologically $h$-cobordant.
\end{itemize}

These examples contrast with Teichner's examples in~\cref{subsection:row-4} in that we have infinitely many, and the stable homeomorphism statuses are different.
A similar phenomenon to the manifolds in $\mathcal{M}(M)$ arises for manifolds homotopy equivalent to $\RP^4 \# \RP^4$~\cite{BDK-07}, except that these manifolds are of course nonorientable.

We next prove the following lemma, which was used in the proof of \cref{intro-thm-1-repeated}.

\begin{lemma}\label{lemma:hom-self-equiv-finite}
The group of homotopy self-equivalences $\hAut(M)$ of $M$ is a finite group.
\end{lemma}

\begin{proof}[Sketch of proof]
  For this we will use the braid of exact sequences from Hambleton-Kreck~\cite{Hambleton-Kreck-HSEs}*{p.~148}, which applies since $M$ is spin.
This braid in particular fits the group of homotopy self-equivalences $\hAut(M)$ into an exact sequence~\cite{Hambleton-Kreck-HSEs}*{Corollary~2.13},
\begin{equation}\label{eq:exact}
\begin{tikzcd}
\widehat{\Omega}_5^{\Spin}(B,M) \arrow[r]	&\hAut(M)\arrow[r]	&\hAut(B)
\end{tikzcd}
\end{equation}
sandwiched between the homotopy automorphisms $\hAut(B)$ of the Postnikov 2-type $B$, and a spin bordism group $\wh{\Omega}_5^{\Spin}(B,M)$ that we shall define below.

Since $\pi_2(M)=0$, the Postnikov 2-type of $M$ is $B := B(\Z/p \times \Z)$. The homotopy classes of homotopy equivalences of $B$ are therefore isomorphic to the automorphisms of the group $\Z/p \times \Z$. This group of automorphisms is a finite group.

Let $\wh{\Omega}_4^{\Spin}(M) \subseteq \Omega_4^{\Spin}(M)$ denote the subset of bordism classes $(X,f)$ where the reference map $f \colon X \to M$ has degree $0$, let $\partial \colon \Omega_5^{\Spin}(B,M) \to \Omega_4^{\Spin}(M)$ be the boundary map in the long exact sequence of the pair, and let
$\wh{\Omega}_5^{\Spin}(B,M) \subseteq \Omega_5^{\Spin}(B,M)$ be $\partial^{-1}(\wh{\Omega}_4^{\Spin}(M))$.
By~\cite{Hambleton-Kreck-HSEs}*{Lemma~2.2}, there is a long exact sequence
\[
\begin{tikzcd}
\Omega_5^{\Spin}(M) \arrow[r] &\Omega_5^{\Spin}(B)  \arrow[r]	&\wh{\Omega}_5^{\Spin}(B,M) \arrow[r]	&\wh{\Omega}_4^{\Spin}(M)  \arrow[r]   &\Omega_4^{\Spin}(B).
\end{tikzcd}\]
The four non-relative spin bordism groups, and the first and last maps in the sequence,  can be computed using the (natural) Atiyah-Hirzebruch spectral sequence for the generalised homology theory of spin bordism. We omit the details, because very similar details will appear below in the proof of \cref{lemma:bound-on-size-of-homotopy-auts}. There we will restrict to $p$ a power of $2$, but  the computation that both the cokernel of the map $\Omega_5^{\Spin}(M) \to \Omega_5^{\Spin}(B)$ and the kernel of the map $\wh{\Omega}_4^{\Spin}(M)  \to   \Omega_4^{\Spin}(B)$ are finite groups is similar for all~$p \geq 2$.  To avoid essentially repeating ourselves, we only give the details in the proof in \cref{subsection:row-14} below, since in that case the result is new, and more precise upper bounds are required.
It follows that $\wh{\Omega}_5^{\Spin}(B,M)$ is finite. Therefore~\eqref{eq:exact} shows that $\hAut(M)$ is finite, as desired.
\end{proof}


\subsection{\texorpdfstring{Simple homotopy equivalent, $h$-cobordant $4$-manifolds that are not $s$-cobordant}{Simple homotopy equivalent, h-cobordant 4-manifolds that are not s-cobordant}}\label{subsection:row-14}

We construct arbitrarily large collections of closed, orientable, topological 4-manifolds that are simple homotopy equivalent and $h$-cobordant but not topologically $s$-cobordant. This will prove \cref{intro-thm-2}, which we restate below. We will employ  the same scheme as in the previous subsection. The manifolds in each collection will be simple homotopy equivalent and $h$-cobordant to a fixed 4-manifold $L_{2^r,1} \times S^1$, for some $r$.  We will show that by making $r$ large enough we can obtain a collection of 4-manifolds of any given size, with the following properties.

\begin{itemize}
  \item They are orientable and have nontrivial fundamental group isomorphic to $\Z/{2^r}\times \Z$ for some $r\geq 8$.
  \item They are all simple homotopy equivalent and topologically $h$-cobordant to one another. As a result, they are homotopy equivalent, stably homeomorphic, and $\CP^2$-stably homeomorphic. They are are pairwise not topologically $s$-cobordant and therefore not homeomorphic.
  \item Since they are homotopy equivalent they all have vanishing Euler characteristic, the same as $L_{2^r,1} \times S^1$.
  \item They are all stably homeomorphic and $\CP^2$-stably homeomorphic.
  \item  We do not know whether they are smoothable, and therefore the smooth questions are not applicable.
\end{itemize}

\begin{reptheorem}{intro-thm-2}
  For every $n\geq 1$, there is a collection ~$\{N_i\}_{i=1}^n$ of closed, orientable, topological 4-manifolds, that are all simple homotopy equivalent and $h$-cobordant to one another, but which are pairwise not $s$-cobordant.
\end{reptheorem}

\begin{proof}
Let $M_r := L_{2^r,1} \times S^1$, for $r\geq 1$.
Let $\pi_r := \Z/2^r\times \Z$, and let $G_r:= \Z/2^r$. The proof will again use the surgery sequence, both with the $h$ and $s$ decorations. We begin by investigating the individual terms.
Shaneson splitting~\cite{Shaneson-splitting} shows that
\[L_5^s(\Z\pi_r) \cong L_5^s(\Z G_r) \oplus L_4^h(\Z G_r) \text{ and } L_5^h(\Z\pi_r) \cong L_5^h(\Z G_r) \oplus L_4^p(\Z G_r).\]
Then by \cite{Bak-computation-L-groups}*{Theorem~7}, we know
\[L_5^s(\Z G_r) = 0 = L_5^h(\Z G_r).\]
Here we use that $s=0$, in the notation of that theorem (this is a different $s$ to the $s$-decoration of the $L$-groups).
By \cite{Bak-computation-L-groups}*{Theorem~1} we have that
\[L_4^p(\Z G_r) = \Z^{m(r)}.\]
where $m(r) := 2^{r-1}+1$.
By \cite{Bak-computation-L-groups}*{Theorem~2} and \cite{CS-torsion-in-L-groups}*{Theorem~B}:
\[L_4^h(\Z G_r) = \Z^{m(r)} \oplus (\Z/2)^{n(r)}\]
where
\[n(r) :=  \lfloor 2(2^{r-2} +2)/3 \rfloor - \lfloor r/2 \rfloor -1.\]
Putting this all together we have:
\[L_5^s(\Z\pi_r) \cong \Z^{m(r)} \oplus (\Z/2)^{n(r)} \text{ and } L_5^h(\Z\pi_r) \cong \Z^{m(r)}.\]
The kernel of the forgetful map $L_5^s(\Z\pi_r) \to L_5^h(\Z\pi_r)$ is the torsion summand
\[K := \ker (L_5^s(\Z\pi_r) \to L_5^h(\Z\pi_r)) \cong (\Z/2)^{n(r)},\]
since $L^h_5(\Z\pi_r)$ is torsion free and the kernel of the map $L^s_n(R)\to L^h_n(R)$ is 2-torsion, for every ring with involution $R$ and every $n\in\Z$ by the Rothenberg exact sequence \cite{Shaneson-splitting}*{Proposition~4.1}.
 The elements in $K$ act on the simple structure set of $M_r$, producing topological manifolds that are stably homeomorphic and simple homotopy equivalent to $M_r$.
We can compute the normal maps as
\begin{align*}
\mathcal{N}(M_r \times [0,1],M_r \times \{0,1\}) &\cong H^4(M_r \times [0,1],M_r \times \{0,1\};\Z) \oplus H^2(M_r \times [0,1],M_r \times \{0,1\};\Z/2) \\
&\cong \Z \oplus \Z/2^r \oplus \Z/2.
\end{align*}
We have a direct sum because there is a 5-connected map $\G/\TOP \to K(\Z,4) \times K(\Z/2,2)$~\cite{Kirby-Taylor}.
The $\Z$ summand is detected by the ordinary signature, and in particular it maps to one of the multisignature summands in $\Z^{m(r)}\subseteq L_5^s(\Z\pi_r)$, under the surgery obstruction map. The torsion summand $\Z/2^r \oplus \Z/2$ could map to the torsion elements in $L_5^s(\Z\pi_r)$.  But at least a summand of $L_5^s(\Z\pi_r)$, one isomorphic to $(\Z/2)^{n(r)-2}$, acts nontrivially on the simple structure set. Note that $|(\Z/2)^{n(r)-2}| = 2^{n(r) -2}$.

Still, it might be the case that some of the elements of the simple structure set $\mathcal{S}^s(M_r)$ obtained by this action of $L_5^s$ are identified by the action of the simple homotopy self-equivalences $\hAut^s(M_r)$.  Note that $|\hAut^s(M_r)| \leq |\hAut(M_r)|$.  We will show the following lemma.
 \begin{lemma}\label{lemma:bound-on-size-of-homotopy-auts}
 Let $r \geq 8$. Then
    \[ 2^{n(r)-2} / |\hAut(M_r)| > 1.\]
    Moreover, for any $k$ there exists an $r$ with $2^{n(r)-2} / |\hAut(M_r)| > k.$
\end{lemma}

Every element $M_r(\kappa) := \kappa \cdot [\Id \colon M_r \to M_r] \in \mathcal{S}^s(M_r)$ arising from the action of an element $\kappa \in K$ maps trivially to the homotopy structure set $\mathcal{S}^h(M_r)$, by the diagram
\[
\begin{tikzcd}
L_5^s(\Z\pi_r) \ar[r,"W"] \ar[d] & \mathcal{S}^s(M_r) \ar[d] \\
L_5^h(\Z\pi_r) \ar[r,"W"] & \mathcal{S}^h(M_r).
\end{tikzcd}
\]
The $M_r(\kappa)$ are all therefore $h$-cobordant to $M_r$.  Since they arise from the action of the simple $L$-group $L_5^s(\Z\pi_r)$, they are all simple homotopy equivalent and stably homeomorphic to one another.
By \cref{lemma:bound-on-size-of-homotopy-auts} there is more than one orbit of $\{M_r(\kappa)\}_{\kappa \in K} / \hAut^s(M_r)$, and these are not $s$-cobordant and therefore not homeomorphic manifolds.
Moreover, for a given $k$ we can choose $r$ so that there are at least $k$ orbits, and therefore we find arbitrarily large collections.
\end{proof}

Now we prove \cref{lemma:bound-on-size-of-homotopy-auts}.

\begin{proof}[Proof of \cref{lemma:bound-on-size-of-homotopy-auts}]
As in \cref{subsection:row-13}, we use the braid from~\cite{Hambleton-Kreck-HSEs}. There we claimed that $\hAut(L_{p,q}\times S^1)$ is finite for any lens space $L_{p,q}$ with $p$ odd.  Now we claim the same when $p=2^r$, and moreover in this case we compute an explicit upper bound, in terms of~$r$, for the order of $\hAut(L_{p,q}\times S^1)$.
The braid includes the exact sequence~\cite{Hambleton-Kreck-HSEs}*{Corollary~2.13}
\begin{equation}\label{eq:SES}
\widehat{\Omega}_5^{\Spin}(B_r,M_r) \to \hAut(M_r) \to \hAut(B_r),
\end{equation}
so we need upper bounds for the cardinalities of $\widehat{\Omega}_5^{\Spin}(B_r,M_r)$ and $\hAut(B_r)$, where $\hAut(B_r)$ denotes the set of homotopy self-equivalences of the Postnikov 2-type $B_r$ up to homotopy. The spin bordism group $\wh{\Omega}_5^{\Spin}(B_r,M_r)$ also appeared in the proof of \cref{lemma:hom-self-equiv-finite}, and we refer the reader there for the definition.

First we compute $\hAut(B_r)$. Since $\pi_2(M_r)=0$, the Postnikov 2-type $B_r$ is given by $B\pi_r$. Therefore, $\hAut(B_r)\cong \Aut(\Z/2^r\times \Z)$. We will now show that $|\hAut(B_r)| = 2^{2r}$. An arbitrary endomorphism of $\Z/2^r\times \Z$ maps $(1,0)$ to $(a, 0)$ and $(0,1)$ to $(b,z)$, for some $a,b\in \Z/2^r$ and $z\in \Z$. For an automorphism, we must have that $z=\pm 1$ and $a$ must be a generator of $\Z/2^r$. Hence there are $2^{2r}=2\cdot 2^{r-1}\cdot 2^{r}$ allowed choices for $z,a$ and $b$.

To find an upper bound for $|\wh{\Omega}_5^{\Spin}(B_r,M_r)|$ we use the exact sequence~\cite{Hambleton-Kreck-HSEs}*{Lemma~2.2}
\begin{equation}\label{eq:LES}
\Omega_5^{\Spin}(M_r) \to \Omega_5^{\Spin}(B_r)  \to \wh{\Omega}_5^{\Spin}(B_r,M_r) \to \wh{\Omega}_4^{\Spin}(M_r)  \to   \Omega_4^{\Spin}(B_r).
\end{equation}
We investigate the bordism groups using the Atiyah-Hirzebruch spectral sequence.
The sequence we need, for $X\in\{M_r,B_r\}$, is
\[E^2_{p,q} = H_p(X;\Omega_q^{\Spin}) \Rightarrow \Omega_{p+q}^{\Spin}(X).\]
In the range of interest $0\leq q\leq 5$, we have:
\[
\Omega_q^{\Spin} \cong
\left\{\begin{array}{lr}
        \Z, & \text{for } q=0,4,\\
        \Z/2, & \text{for } q=1,2,\\
        0, & \text{for } q=3,5.\\
        \end{array}\right.
\]
We also need the homology of~$M_r$, which by the K\"{u}nneth theorem with $\Z/2$-coefficients is as follows:
\[H_k(M_r;\Z/2) \cong \begin{cases}
  \Z/2, & \text{for }k=0,4, \\
  (\Z/2)^2, & \text{for }k=1,2,3, \\
  0, & \text{otherwise.}
\end{cases}\]
Additionally $H_1(M_r;\Z) \cong \Z/2^r \oplus \Z$.  Since $B_r = B\pi_r$ can be constructed from $M_r$ by adding cells of dimension four and higher, for $A\in\{\Z/2,\Z\}$ the induced map
\[H_k(M_r;A) \to H_k(B_r;A)\]
is an isomorphism for $k=0,1,2$ and a surjection for $k=3$.
Finally we will need that $H_5(B_r;\Z) \cong \Z/2^r$.

Now we use this homology information together with the spectral sequences to obtain an upper bound for the cardinality of the cokernel of the map $\Omega_5^{\Spin}(M_r) \to \Omega_5^{\Spin}(B_r)$ from~\eqref{eq:LES}. The map $M_r\to B_r$ induces maps between each page of the spectral sequences computing $\Omega_5^{\Spin}(M_r)$ and $\Omega_5^{\Spin}(B_r)$.
The nonzero terms $E^2_{p,q}$ on the $E^2$ page with $p+q=5$ are as follows:
\[H_1(-;\Z), H_3(-;\Z/2) , H_4(-;\Z/2), \text{ and } H_5(-;\Z).\]
The maps $H_1(M_r;\Z) \to H_1(B_r;\Z)$ and $H_3(M_r;\Z/2) \to H_3(M_r;\Z/2)$ are onto as explained above, so by naturality of the spectral sequence these terms do not contribute to the cokernel.
The mod 2 fundamental class in $H_4(M_r;\Z/2)\cong \Z/2$ maps nontrivially to $H_4(B_r;\Z/2) \cong (\Z/2)^2$, so possibly one $\Z/2$ could contribute to the cokernel (whether or not it does so depends on differentials  which we shall not take into account).  The only other contribution to the cokernel comes from the term $H_5(B_r;\Z) \cong \Z/2^r$.
As a result the cokernel of $\Omega_5^{\Spin}(M_r) \to \Omega_5^{\Spin}(B_r)$
has at most~$2^{r+1}$ elements.

Next we find an upper bound on the size of
$\ker\big(\wh{\Omega}_4^{\Spin}(M_r)  \to   \Omega_4^{\Spin}(B_r)\big)$. We do this by considering the composition
\[\wh{\Omega}_4^{\Spin}(M_r)  \to \Omega_4^{\Spin}(M_r) \to   \Omega_4^{\Spin}(B_r).\]
Consider the Atiyah-Hirzebruch spectral sequence computing $\Omega_4^{\Spin}(M_r)$. The nonzero terms $E^2_{p,q}$ on the $E^2$ page with $p+q=4$ are
\[H_0(M_r;\Z), H_2(M_r;\Z/2) , H_3(M_r;\Z/2), \text{ and } H_4(M_r;\Z).\]
The map $H_0(M_r;\Omega_4^{\Spin}) \cong \Z \to H_0(B_r;\Omega_4^{\Spin}) \cong \Z$ is an isomorphism, as explained above.
The image of the inclusion $\wh{\Omega}_4^{\Spin}(M_r)  \to \Omega_4^{\Spin}(M_r)$ consists of elements with trivial image under the edge homomorphism $\Omega_4^{\Spin}(M_r)\to H_4(M_r;\Z) \cong E^2_{p,0}$ term on the $E^2$ page, since the latter map is given by the mapping degree times the fundamental class $[M_r]$.

It follows that the kernel of $\wh{\Omega}_4^{\Spin}(M_r)  \to   \Omega_4^{\Spin}(B_r)$ is generated by elements coming from the terms $H_2(M_r;\Z/2) \cong (\Z/2)^2$ and $H_3(M_r;\Z/2)\cong (\Z/2)^2$, and so the kernel has at most~$2^4$ elements.
Thus by~\eqref{eq:LES}, we see that
\[|\widehat{\Omega}_5^{\Spin}(B_r,M_r)| \leq 2^{r+1} \cdot 2^4 = 2^{r+5}.\]
It now follows from the sequence~\eqref{eq:SES} that
\[|\hAut(M_r)| \leq |\hAut(B_r)| \cdot |\widehat{\Omega}_5^{\Spin}(B_r,M_r)| \leq  2^{2r} \cdot 2^{r+5} = 2^{3r+5}.\]
An elementary calculation, recalling that $n(r) =  \lfloor 2(2^{r-2} +2)/3 \rfloor - \lfloor r/2 \rfloor - 1$, shows that for $r \geq 8$ we have
\[n(r) -2 - 3r-5 >0.\]
This implies that $2^{n(r)-2} / |\hAut(M_r)| \geq 2^{n(r)-2}/2^{3r-5} = 2^{n(r) - 2 - 3r - 5} > 1$, as desired.
In addition, note that $n(r) -2 - 3r-5 \to \infty$ as $r \to \infty$. It follows that for a given $k$, there exists an $r$ such that $2^{n(r)-2} / |\hAut(M_r)| \geq  2^{n(r) - 2 - 3r - 5} > k$.
 \end{proof}

\addcontentsline{toc}{section}{Table of counterexamples in 4-manifold topology}
\begin{center}
\begin{sidewaystable}
\ra{1.5}
    \begin{tabular}{@{}p{7mm}lp{5mm}p{5mm}p{5mm}p{2mm}p{7mm}p{7mm}p{7mm}p{7mm}p{7mm}p{7mm}p{7mm}p{7mm}p{7mm}p{7mm}p{7mm}p{7mm}p{7mm}p{7mm}}
        \toprule
        &	Examples	& \multicolumn{3}{c}{Properties} &	& \multicolumn{13}{c}{Equivalence relations}	&\\
        \cmidrule{3-5} \cmidrule{7-20}
   		&	&\tilt{smooth}	&\tilt{oriented}	&\tilt{$\pi_1=1$}	&	&\tilt{equal $\chi$}	 &\tilt{$S^2\times S^2$-stably homeo.}	 &\tilt{$\CP^2$-stably homeo.}	&\tilt{$S^2\times S^2$-stably diffeo.}	 &\tilt{$\CP^2$-stably diffeo.}	&\tilt{homotopy equiv.}	 &\tilt{simple homotopy equiv.}	&\tilt{top.\ $h$-cobordant}	&\tilt{top.\ $s$-cobordant} &\tilt{smoothly $h$-cobordant}	&\tilt{smoothly $s$-cobordant}	 &\tilt{homeomorphic}	&\tilt{diffeomorphic}	 &\\
        \midrule
        \S\ref{subsection:row-1}& $S^4$ and $S^2\times S^2$		&\cmark	&\cmark	&\cmark	&&\xmark	&\cmark	 &\cmark	&\cmark	&\cmark	&\xmark	&\xmark & \xmark  & \xmark	 &\xmark	&\xmark	&\xmark	&\xmark	&\\
        \S\ref{subsection:row-2}& $S^2 \times S^2$ and $S^2 \mathbin{\wt{\times}} S^2$		& \cmark &\cmark	&\cmark	 &&\cmark	&\xmark	&\cmark	&\xmark	 &\cmark	 &\xmark	&\xmark	& \xmark  & \xmark &\xmark	&\xmark	&\xmark	 &\xmark	& \\
        \S\ref{subsection:row-3}& $\CP^2$ and $\star\CP^2$		& \xmark & \cmark	&\cmark	&&\cmark	&\xmark	 &\xmark	& n/a	& n/a	&\cmark	 &\cmark & \xmark  & \xmark	& n/a	& n/a	&\xmark	& n/a	&\\
        \S\ref{subsection:row-4}& $\RP^4\# \CP^2$ and $\mathcal{R}\#\star\CP^2$		&\cmark	&\xmark	 &\xmark	&&\cmark	&\cmark	& \cmark	 &\cmark	 &\cmark	&\cmark	 &\cmark	& \xmark  & \xmark &\xmark	 &\xmark	&\xmark	&\xmark	&\\
        \S\ref{subsection:row-5}& $K3 \# \RP^4$ and $\bighash{11} (S^2 \times S^2) \# \RP^4$		& \cmark	&\xmark	 &\xmark	&&\cmark	&\cmark	& \cmark	 &\xmark	& \cmark	 &\cmark	&\cmark	& \cmark & \cmark & \xmark	 &\xmark	&\cmark	&\xmark	&\\
        \S\ref{subsection:row-6}& $\RP^4$ and $R$		&\cmark	&\xmark	&\xmark	&&\cmark	&\cmark	&\cmark	 &\xmark	&\cmark	&\cmark	&\cmark	& \cmark & \cmark  & \xmark	&\xmark	 &\cmark	&\xmark	&\\
        \S\ref{subsection:row-7}& \makecell{$L_1\times S^1,\dots, L_k\times S^1,$ with \\ $L_i\simeq L_j$ but $L_i\not\cong L_j$ for $i\neq j$}
                		&\cmark	&\cmark	&\xmark	&&\cmark	& \cmark	 &\cmark	& \cmark	& \cmark &\cmark	 & \cmark	 & \xmark  & \xmark & \xmark & \xmark		& \xmark	 &\xmark	&\\
        \S\ref{subsection:row-8}& $E(1)$ and $E(1)_{2,3}$		& \cmark	&\cmark	&\cmark	&&\cmark	&\cmark	& \cmark	&\cmark	& \cmark	&\cmark	 &\cmark & \cmark & \cmark	&\cmark	&\cmark	&\cmark	&\xmark	&\\
         \S\ref{subsection:row-9}& $\bighash{3} E_8$ and $Le$	&\xmark	&\cmark	 &\cmark	&&\cmark	&\cmark	& \cmark	 & n/a	 & n/a & \xmark	 &\xmark & \xmark  & \xmark	& n/a	 & n/a	&\xmark	& n/a	&\\

        \S\ref{subsection:row-10}& Kreck--Schafer manifolds		&\cmark	&\cmark	&\xmark	&&\cmark	&\cmark	 &\cmark	&\cmark	&\cmark	&\xmark	&\xmark	 & \xmark &\xmark & \xmark & \xmark	&\xmark	&\xmark	&\\
         \S\ref{subsection:row-11}& \makecell{Teichner's $E\# E  \# \bighash{k} (S^2 \times S^2)$\\ and $\star E \# \star E  \# \bighash{k} (S^2 \times S^2)$} 
&\cmark	&\cmark	 &\xmark	&&\cmark	&\xmark	& \cmark	 &\xmark	 & \cmark & \cmark	 &\cmark & \xmark  & \xmark	& \xmark	 &\xmark	 &\xmark	 &\xmark	&\\
        \S\ref{subsection:row-12}& Akbulut's $P$ and $Q$		&\cmark	& \xmark	&\xmark	&&\cmark	& \cmark	& \cmark	& \xmark	& \cmark	 &\cmark	&\cmark & \cmark  & \cmark	&\xmark	& \xmark	&\cmark	&\xmark	&\\
		\S\ref{subsection:row-13}&  $\mathcal{M}(L_{p,q} \times S^1)$, $p$ odd, $\infty$ set    		
&?	&\cmark	 &\xmark	&&\cmark	&\cmark	 & \cmark	 & n/a	& n/a & \cmark	 &\cmark & \xmark  & \xmark	& n/a	& n/a	& \xmark	 & n/a	&\\
		\S\ref{subsection:row-14}& $\{M_r(\kappa)\}_{\kappa \in K}$    		
&?	&\cmark	&\xmark	    &&\cmark	&\cmark	 & \cmark	 & n/a	& n/a & \cmark	 &\cmark & \cmark  & \xmark	& n/a	& n/a	& \xmark	 & n/a	&\\
        \bottomrule
    \end{tabular}
\caption{Counterexamples in $4$-manifold topology.}\label{table}
\end{sidewaystable}
\end{center}

\bibliography{bib}

\end{document}